\documentclass[11pt, a4paper]{amsart}
\usepackage{amsfonts}
\usepackage{amsmath}
\usepackage{amssymb,xcolor}
\usepackage{tikz-cd}
\usepackage{bbm}
\usepackage{amscd}
\usepackage[left=2.5cm,right=2.5cm,bottom=3cm,top=3cm]{geometry}
\usepackage{etoolbox}
\usepackage{filecontents}
\usepackage{enumitem}
\usepackage[hidelinks]{hyperref}

\AtBeginDocument{%
   \def\MR#1{}
}

\newtheorem{theorem}{Theorem}
\newtheorem{lemma}[theorem]{Lemma}
\newtheorem{proposition}[theorem]{Proposition}

\newtheorem{conjecture}[theorem]{Conjecture}
\newtheorem{corollary}[theorem]{Corollary}

\theoremstyle{definition}
\newtheorem{definition}[theorem]{Definition}
\newtheorem{example}[theorem]{Example}

\newtheorem{question}[theorem]{Question}

\theoremstyle{remark}
\newtheorem{remark}[theorem]{Remark}

\def\Sym{\mathrm{Sym}}

\begin{document}
\title{Quantifying metric approximations of discrete groups}

\author{Goulnara Arzhantseva}
\address{Universit\"at Wien, Fakult\"at f\"ur Mathematik\\
Oskar-Morgenstern-Platz 1, 1090 Wien, Austria.}
\email{goulnara.arzhantseva@univie.ac.at}

\author{Pierre-Alain Cherix}
\address{Universit\'{e} de Gen\`{e}ve,
Section de Math\'{e}matiques, 2-4 rue du Li\`{e}vre, Case postale
64, 1211 Gen\`{e}ve 4, Switzerland.
}
\email{pierre-alain.cherix@unige.ch}

\date{}

\thanks{This research was partially supported by the European Research Council (ERC) grant of Goulnara Arzhantseva, ``ANALYTIC'' grant agreement no.\ 259527. } 
\baselineskip=16pt

\begin{abstract}
We introduce and systematically study a profile function whose asymptotic behavior
quantifies the dimension or the size of a metric approximation of a finitely generated group $G$ by a family of groups $\mathcal{F}=\{(G_{\alpha},d_{\alpha}, k_{\alpha}, \varepsilon _{\alpha })\}_{\alpha\in I, }$ where each group $G_\alpha$ is equipped with a bi-invariant metric $d_{\alpha}$ and a dimension $k_{\alpha}$, for strictly positive real numbers $\varepsilon _{\alpha }$ such that $%
\inf_{\alpha }\varepsilon _{\alpha }>0$. 
Through the notion of a residually amenable profile that we introduce, our approach generalizes classical isoperimetric (aka F\o lner) profiles of amenable groups and recently introduced functions quantifying residually finite groups. Our viewpoint is much more general and covers hyperlinear and sofic approximations as well as many other metric approximations such as weakly sofic, weakly hyperlinear, and linear sofic approximations. 
\end{abstract}
\maketitle

\section{Introduction}
Approximation is ubiquitous in mathematics. In the theory of groups, it is particularly natural to approximate
infinite groups by finite ones. A fundamental realization of this idea has lead Mal'cev (1940's) and  P. Hall (1955) to the notion of a residually finite group:
a group where the algebraic structure on any finite fixed set of elements is exactly as if these elements
were in a suitable finite quotient of the group.

Once a concept of approximation is coined, a crucial question is how to compare distinct approximations of the same object, and, in particular,
how to quantify the way an object is approximated. For residually finite groups, there are two main  ways of 
quantifying the approximation of an infinite group by finite ones. The first way is to compute how many subgroups of a given finite index
the group possesses. This is a classical subject of research on the subgroup growth, initiated by M. Hall (1949), which allows to enumerate how the group can be approximated by a finite quotient of a prescribed cardinality. The second way of quantifying is to compute the minimal cardinality among all possible finite quotients that detect  the algebraic structure of
the fixed finite set of elements of the residually finite group. This viewpoint is more recent and it is about the so-called full residual finiteness growth, see below for the definition.

In this paper, we push this second idea of quantifying of approximations of infinite groups significantly beyond the class of residually finite groups and
apply it to much more general \emph{metric} approximations of infinite groups in contrast to classical \emph{algebraic} approximations.
Metric approximations are approximations by groups equipped with bi-invariant metrics (see the next section for precise definitions)
and they are very natural to study. Intuitively, we require that the algebraic operation on a finite set of group elements of the approximated group is almost
as if these elements were in the approximating group, where `almost' refers to the fixed bi-invariant metric. This simple idea has gained a major importance following Gromov's
introduction of \emph{sofic} groups (= groups metrically approximated by symmetric groups of finite degrees, endowed with the normalized Hamming distance)
 and his settlement, for sofic groups, of Gottschalk's surjunctivity conjecture~(1973) in topological dynamics.  
Another renowned example of metric approximation is that by unitary groups of finite rank, endowed with the normalized Hilbert-Schmidt distance.
This defines the class of \emph{hyperlinear} groups, appeared in the context of Connes' embedding problem~(1972) in operator algebra.

We encompass both sofic and hyperlinear groups  as well as their generalizations such as linear sofic groups, weakly sofic groups, and weakly hyperlinear groups into a general framework of metric approximations by groups with, in addition to a prescribed bi-invariant metric, a dimension or a size, associated with each of the approximating groups. For instance, the dimension of a finite symmetric group is chosen to be its degree, of a unitary group -- its rank, of a finite group -- its cardinality, etc. Our general quantification function, called \emph{metric profile}, is then defined to be, given a finite set of group elements in the approximated group, e.g. the ball of finite radius with respect to the word length metric, the minimal dimension among all possible metric approximations which `almost' preserve the algebraic structure of  this finite set. Viewed within sofic groups, our approach is orthogonal to the recently emerged theory of sofic entropy started in the seminal work of L. Bowen (such a theory is not yet available for an a priori wider class of hyperlinear groups). Restricted to residually finite groups, the contrast between Bowen's viewpoint and our approach is exactly the distinction 
between the subgroup growth of a group and the full residual finiteness growth, respectively.

Since metric approximations generalize classical algebraic approximations,
the previously known functions, quantifying `exact' approximations (versus `almost' ones), 
occur to be upper bounds for our metric profile. For example, a knowledge about the full residual finiteness growth of a residually finite group
gives an estimate on the sofic and on the hyperlinear profiles of such a group.
If the approximating groups are amenable,
then besides a chosen dimension, they carry an associated isoperimetric function, the famous F\o lner function.
We make use of this classical function and of our metric profile philosophy to define
the \emph{residually amenable profile} for every residually amenable group (and more generally, for every group locally embeddable into
amenable ones). This allows to extend a classical study of F\o lner functions of amenable groups to
non-amenable groups metrically approximable by amenable ones.

The main aim of this paper is to provide a necessary theoretical base for a further more specific quantitative analysis of metric approximations of concrete discrete groups.
We meticulously compare our metric profile with previously investigated quantifying functions alluded to above.
Since the classes of groups we study are preserved under several group-theoretical operations such as taking subgroups, direct and free products,
extensions by amenable groups,  restricted wreath products, etc.,
we also provide the corresponding estimates on the suitable metric profiles.
On the way, we collect some crucial examples and finally formulate a number of open problems.\smallskip

{\bf Acknowledgment.} The main concepts of this paper have been introduced in 2008. Since then this work has been presented 
on several occasions at the universities of Neuch\^atel, Copenhagen, Aix-Marseille, at
the ENS Lyon, ETH Z\"urich, MF Oberwolfach, and CRM Barcelona. The first author is
grateful to colleagues at these institutions for the hospitality.

The authors thank Martino Lupini for his help on the final version of this text.

\section{$\mathcal{F}$-approximations and $\mathcal{F}$-profile}

Let $I$ denote an \emph{index set}. We let $\mathcal{F}=\left(
G_{\alpha }, d_{\alpha}, k_{\alpha },\varepsilon _{\alpha }\right) _{\alpha \in I}$,
where $G_{\alpha }$ is a  group with a bi-invariant distance $d_{\alpha
} $ and identity element $e_{\alpha }$, $k_{\alpha }$ is a natural number
that can be thought as the \emph{dimension }of $G_{\alpha }$, and $%
\varepsilon _{\alpha }$ is a strictly positive real number such that $%
\inf_{\alpha }\varepsilon _{\alpha }>0$. 

Let $G$ be a countable discrete
group with identity $e_{G}$ and a distinguished generating set $S\subseteq G$. We
denote by $\left\vert g\right\vert _{S}$ the length of an element $g\in G$
with respect to the word length metric defined by $S$. We let $B_{G,S}\left(
n\right) $ be the ball of center $e_{G}$ and radius $n$ with respect to the
word length metric induced by $S$. 

\begin{definition}[Approximation]\label{def:app}
Let $n\in \mathbb{N}$ and $\varepsilon_{\alpha }
>0$. An \emph{$\left( n,\varepsilon _{\alpha }\right) $-approximation} of $\left(
G,S\right) $ by a group $G_{\alpha }$ is a function $\pi\colon G\rightarrow G_{\alpha }$
satisfying the following:

\begin{enumerate}
\item[(1)] $d_{\alpha }\left( \pi \left( g\right) \pi \left( h\right) ,\pi \left(
gh\right) \right) <1/n$ for every $g,h\in B_{G,S}\left( n\right) $ with $gh\in B_{G,S}\left( n\right)$, and

\item[(2)] $d_{\alpha }\left( \pi \left( g\right) ,\pi \left( h\right) \right)
>\varepsilon _{\alpha }-1/n$ for every $g,h\in B_{G,S}\left( n\right) $ such
that $g\neq h$.
\end{enumerate}
Such an $\left( n,\varepsilon _{\alpha }\right) $-approximation is said to be \emph{of dimension $k_{\alpha }.$}

An \emph{$\mathcal{F}$-approximation} of $\left( G,S\right) $ is a sequence $\left(
\pi _{n}\right)_{n\in \mathbb{N}} $ such that, for every $n\in \mathbb{N}$, $\pi _{n}$ is an $%
\left( n,\varepsilon _{\alpha }\right) $-approximation of $\left( G,S\right) 
$ by $G_{\alpha }$ for some $\alpha \in I$. 

A finitely generated group $G$
is \emph{$\mathcal{F}$-approximable} if $\left( G,S\right) $ admits an $\mathcal{F}$%
-approximation for some (or, equivalently, any) finite generating set $%
S\subseteq G$.
\end{definition}

The above condition (1) is often called `almost homomorphism on the ball' while
condition (2) is termed as `uniform injectivity'.

In the definition below we convene that the minimum of the
empty set is $+\infty $.

\begin{definition}[Profile and dimension]\label{def:profile}
Let $G$ be a finitely generated group with a finite
generating set $S$. The \emph{$\mathcal{F}$-profile} of $G$ is the function $%
\mathcal{D}_{G,S}^{\mathcal{F}}\colon\mathbb{N}\rightarrow \mathbb{N}\cup \left\{
+\infty \right\} $ defined by\footnote{We set $\min\{\emptyset\}=+\infty$.}
$$
\mathcal{D}_{G,S}^{\mathcal{F}}\left( n\right)=\min\left\{ {k\in \mathbb{N}} \mid \exists \hbox{ an $\left(
n,\varepsilon _{\alpha }\right) $-approximation of $G$ by $%
G_{\alpha }$ of dimension $k_{\alpha }=k$}\right\}.
$$
The \emph{$\mathcal{F}$-dimension} of $G$
is defined by
$$
\dim_{G,S}^{\mathcal{F}}=\limsup_{n\rightarrow +\infty}\frac{1}{n}\log \mathcal{D}%
_{G,S}^{\mathcal{F}}\left( n\right). 
$$
\end{definition}

Observe that $G$ is $\mathcal{F}$-approximable if and only if the function $%
\mathcal{D}_{G,S}^{\mathcal{F}}$ is everywhere finite.  We write simply $
\mathcal{D}_{G,S}$ when the family is irrelevant.

\begin{remark}
One can consider families $\mathcal{F}$ as above where $d_{\alpha }$ is not
necessarily a metric but just a bi-invariant \emph{pseudometric}. One can
always reduce to the case of bi-invariant metric (rather than pseudometric)
by replacing $\left( G_{\alpha },d_{\alpha }\right) $ with $\left( G_{\alpha
}/N_{\alpha },\overline{d}_{\alpha }\right)$, where $N_{\alpha }$ is the
normal subgroup $\left\{ g\in G_{\alpha }:d_{\alpha }\left( g, e_\alpha\right)
=0\right\} $ and $\overline{d}_{\alpha }$ is the bi-invariant metric induced
by $d_{\alpha }$ on the quotient.
\end{remark}

We consider the
quasi-order $\preccurlyeq $ for functions $\mathcal{D}_{1},\mathcal{D}_{2}\colon %
\mathbb{N}\rightarrow \mathbb{N}$ defined by $\mathcal{D}_{1}\preccurlyeq 
\mathcal{D}_{2}$ if and only if there exists a constant $C\in \mathbb{N}$
such that $\mathcal{D}_{1}\left( n\right) \leqslant CD_{2}\left( Cn\right) $ for
every $n\in \mathbb{N}$. We also let $\simeq $ be the equivalence relation
associated with the quasi-order $\preccurlyeq $. Thus $\mathcal{D}_{1}\simeq 
\mathcal{D}_{2}$ iff $\mathcal{D}_{1}\preccurlyeq \mathcal{D}_{2}$ and $%
\mathcal{D}_{2}\preccurlyeq \mathcal{D}_{1}$.

If $S,S^{\prime }$ are two finite-generating sets of $G$, then there exists $%
C\in \mathbb{N}$ such $S\subseteq B_{G,S^{\prime }}\left( C\right) $ and $%
S^{\prime }\subseteq B_{G,S}\left( C\right) $. Therefore $B_{G,S}\left(
n\right) \subseteq B_{G,S^{\prime }}\left( Cn\right) $ and $B_{G,S^{\prime
}}\left( n\right) \subseteq B_{G,S}\left( Cn\right) $ for every $n\in \mathbb{N%
}$. This easily implies that $\mathcal{D}_{G,S}\left( n\right) \leqslant \mathcal{%
D}_{G,S^{\prime }}\left( Cn\right) $ and $\mathcal{D}_{G,S^{\prime }}\left(
n\right) \leqslant \mathcal{D}_{G,S}\left( Cn\right) $. In particular, the $\simeq 
$-equivalence class of $\mathcal{D}_{G,S}$ does not depend on the choice of the finite
generating set $S$. We denote such an equivalence class by $\mathcal{D}_{G}$.

\begin{lemma}
\label{Lemma:approximate-homomorphism} Let $G$ be a group and $F\subseteq G$ be
a finite symmetric subset containing the identity $e_{G}$ of $G$. Let $\varepsilon >0$ and
$H$ be a bi-invariant metric group with metric $d_{H}$ and $\pi
\colon G\rightarrow H$ be a function such that $d_{H}\left( \pi \left( g\right)
\pi \left( h\right) ,\pi \left( gh\right) \right) <\varepsilon $ for every $%
g,h\in F$. 

For every $n\in \mathbb{N}$ and $g_{1},\ldots ,g_{n}\in G$ such
that $g_{1}^{\varepsilon _{1}}\cdots g_{k}^{\varepsilon _{k}}\in F$ for
every $\varepsilon _{1},\ldots ,\varepsilon _{n}\in \left\{ +1,-1\right\} $
and $k\in \left\{ 1,2,\ldots ,n\right\} $ the following assertions hold:

\begin{enumerate}
\item $d_{H}\left( \pi \left( e_{G}\right) ,e_{H}\right) <\varepsilon $;

\item $d_{H}(\pi \left( g^{-1}\right) ,\pi \left( g\right)
^{-1})<2\varepsilon $;

\item $d_{H}\left( \pi \left( g_{1}\cdots g_{n}\right) ,\pi \left(
g_{1}\right) \cdots \pi \left( g_{n}\right) \right) <\left( n-1\right)
\varepsilon $, whenever $n>1$;

\item $d_{H}\left( \pi \left( g_{1}^{\varepsilon _{1}}\right) \cdots \pi
\left( g_{n}^{\varepsilon _{n}}\right) ,\pi \left( g_{1}\right)
^{\varepsilon _{1}}\cdots \pi \left( g_{n}\right) ^{\varepsilon _{n}}\right)
<2n\varepsilon $;

\item $d_{H}\left( \pi \left( g_{1}^{\varepsilon _{1}}\cdots
g_{n}^{\varepsilon _{n}}\right) ,\pi \left( g_{1}\right) ^{\varepsilon
_{1}}\cdots \pi \left( g_{n}\right) ^{\varepsilon _{n}}\right) <\left(
3n-1\right) \varepsilon $.
\end{enumerate}
\end{lemma}

\begin{proof}
Since $e_{G}\in F$, using the hypothesis on $\pi $ and the bi-invariance of $%
d_{H}$ we have 
\begin{equation*}
d_{H}\left( \pi \left( e_{G}\right) ,e_{H}\right) =d_{H}(\pi \left(
e_{G}\right) ^{2},\pi \left( e_{G}\right) )<\varepsilon .
\end{equation*}%
This proves (1). Similarly, for (2) we have%
\begin{equation*}
d_{H}(\pi \left( g^{-1}\right) ,\pi \left( g\right) ^{-1})=d_{H}(\pi \left(
g\right) \pi \left( g^{-1}\right) ,e_{H})\leqslant d_{H}\left( \pi \left(
e_{G}\right) ,e_{H}\right) +\varepsilon <2\varepsilon \text{.}
\end{equation*}%
On can easily prove (3) by induction using the hypothesis, and (4) using
(2). Finally (5) follows from (3) and (4).
\end{proof}

In the following we suppose that $G$ is a finitely generated group with a
presentation $\left\langle X \mid R\right\rangle $ given by generators $X$ and relators $R$. We denote by $S$ the
symmetric generating set of $G$ associated with $X$, i.e. $S=X\sqcup X^{-1}$. Without loss of generality,
we assume that $R=R^{-1}$, i.e. $R$ contains inverses of all of its elements.
We denote by $%
F_{X} $ the free group over the alphabet $X$. For a word $w$ over the letters from $X\sqcup X^{-1}$ we let $%
\left\vert w\right\vert _{F_{X}}$ be its length. In this situation, in order
to define a homomorphism from $G$ to another group $H$, it is often
convenient to define a homomorphism from $F_{X}$ to $H$, in such a way that
any word in $R$ is mapped to the identity element $e_H\in H$. This defines a unique
group homomorphism from $G$ to $H$, and any group homomorphism arises in this
way. In this spirit, we define natural variants of the $\mathcal{F%
}$-profile as follows (recall our convention that the minimum of the empty
set is $+\infty $):

\begin{itemize}
\setlength\itemsep{8pt}
\item $\mathcal{W}_{G,S}^{\mathcal{F}}\left( n\right) $ is the least $k\in 
\mathbb{N}$ such that for some $\alpha $ with $k_{\alpha }=k$ there exists a
homomomorphism $\varphi\colon F_{X}\rightarrow G_{\alpha }$ such that for any
word $w\in F_{X}$ of length at most $n$ one has that $d_{\alpha }\left(
\varphi \left( w\right) ,e_{\alpha }\right) >\varepsilon _{\alpha }$ if $w$
does not represent the identity of $G$ and $d_{\alpha }\left( \varphi
\left( w\right) ,e_{\alpha }\right) <1/n$ otherwise;
\item $\mathcal{R}_{G,S}^{\mathcal{F}}\left( n\right) $ is the least $k\in 
\mathbb{N}$ such that for some $\alpha $ with $k_{\alpha }=k$ there exists a
homomomorphism $\varphi\colon F_{X}\rightarrow G_{\alpha }$ such that for any
word $w\in F_{X}$ of length at most $n$ that does not represent the identity
of $G$ one has that $d_{\alpha }\left( \varphi \left( w\right) ,e_{\alpha
}\right) >\varepsilon _{\alpha }$, and $d_{\alpha }\left( \varphi \left(
r\right) ,e_{\alpha }\right) <1/n$ for any relator $r\in R$ of length at most $n$.
\end{itemize}

In the following proposition, we establish precise quantitative
relations among the notions of profile $\mathcal{D}_{G,S}^{\mathcal{F}},%
\mathcal{W}_{G,S}^{\mathcal{F}},\mathcal{R}_{G,S}^{\mathcal{F}}$ we have just introduced. 
To begin with, we recall some notions from combinatorial group
theory. 

If $w\in F_{X}$ is a word
that represents the identity of $G=\left\langle X \mid R\right\rangle$, then the corresponding combinatorial
\emph{area} $\mathcal{A}_{G,S}\left( w\right) $ is defined to be the least $\ell $
such that $w$ can be written as a product of $\ell $ conjugates of relators
from $R$. The \emph{Dehn function}, denoted by $Dehn_{G,S}\left( m\right) $, is defined to be
the largest combinatorial area $\mathcal{A}_{G,S}\left( w\right)$, where $w$
ranges among all the words of length at most $m$ representing the identity
of $G$. We let $N_{G,S}\left( m\right) $ be the minimum, among all
representations $w=\left( \eta _{1}r_{1}\eta _{1}^{-1}\right) \cdots \left(
\eta _{\ell }r_{\ell }\eta _{\ell }^{-1}\right) $ as product of $\ell \leqslant
Dehn_{G,S}\left( m\right) $ conjugates of relators $r_{1},\ldots ,r_{\ell }\in R$, of the maximum of the length of $r_{1},\ldots ,r_{\ell }$.

\begin{proposition}\label{Prop:DWR}
Under the notation above, the following relations between the functions $%
\mathcal{D}_{G,S}^{\mathcal{F}},$ $\mathcal{W}_{G,S}^{\mathcal{F}},\mathcal{R}%
_{G,S}^{\mathcal{F}}$ hold for all sufficiently large $m$'s:

\begin{enumerate}
\item $\mathcal{W}_{G,S}^{\mathcal{F}}\left( m\right) \leqslant \mathcal{D}%
_{G,S}^{\mathcal{F}}\left( 3m^2\right) $;

\item $\mathcal{D}_{G,S}^{\mathcal{F}}\left( m\right) \leqslant \mathcal{W}%
_{G,S}^{\mathcal{F}}\left( 3m\right) $;

\item $\mathcal{W}_{G,S}^{\mathcal{F}}\left( m\right) \leqslant \mathcal{R}%
_{G,S}^{\mathcal{F}}\left( \max \left\{ Dehn_{G}\left( m\right)
m,N_{G,S}\left( m\right) \right\} \right) $;

\item $\mathcal{R}_{G,S}^{\mathcal{F}}\left( m\right) \leqslant \mathcal{W}%
_{G,S}^{\mathcal{F}}\left( m\right) $.
\end{enumerate}

In particular, if one of the functions $\mathcal{D}_{G,S}^{\mathcal{F}},%
\mathcal{W}_{G,S}^{\mathcal{F}},\mathcal{R}_{G,S}^{\mathcal{F}}$ is
everywhere finite, then all the others are everywhere finite.
\end{proposition}

\begin{proof}
We prove the nontrivial inequalities below.

(1): Fix $m\in \mathbb{N}$ and consider $n=3m^2$. If $k=%
\mathcal{D}_{G,S}^{\mathcal{F}}\left( n\right) $, then for some $\alpha $
with $k_{\alpha }=k$ there exists an $\left( n,\varepsilon _{\alpha }\right) 
$-approximation $\pi\colon G\rightarrow G_{\alpha }$. Let $\varphi
\colon F_{X}\rightarrow G_{\alpha }$ be the homomorphism induced by the
restriction of $\pi $ to $X$. Suppose that $w=x_{1}\cdots x_{l}$ is a word
in $F_{X}$ of length $l\leqslant m$. If $w$ represents the identity in $G$, then by the triangle inequality, and by (1) and (5) of Lemma \ref{Lemma:approximate-homomorphism} we have
that 
\begin{equation*}
d_{\alpha }\left( \varphi \left( w\right) ,e_{\alpha }\right) \leqslant d_{\alpha
}\left( \varphi \left( w\right) ,\pi \left( w\right) \right) +d_{\alpha }\left( \pi
\left( w\right) ,e_{\alpha }\right) <\left( 3m-1\right) /n+1/n=1/m\text{.}
\end{equation*}%
Suppose now that $w$ represents a nontrivial element $g$ of $G$. Then we have
again by the triangle inequality together with (1) and (5) of Lemma \ref{Lemma:approximate-homomorphism}%
\begin{equation*}
d_{\alpha }\left( \varphi \left( w\right) ,e_{\alpha }\right) \geqslant d_{\alpha
}\left( \pi \left( g\right) ,\pi \left( e_{G}\right) \right) -d_{\alpha }\left( \pi
\left( e_{G}\right) ,e_{\alpha }\right) -d_{\alpha }\left( \pi \left( g\right)
,\varphi \left( w\right) \right) >\varepsilon _{\alpha }-1/m\text{.}
\end{equation*}%
(2): Suppose that $m\in \mathbb{N}$ and set $n=3m$. If $k=\mathcal{W}_{G,S}^{%
\mathcal{F}}\left( n\right) $, then for some $\alpha $ with $k_{\alpha }=k$
there exists a homomorphism $\varphi\colon F_{X}\rightarrow G_{\alpha }$ such
that for any word $w\in F_{X}$ of length at most $n$ one has that $d_{\alpha
}\left( \varphi \left( w\right) ,e_{\alpha }\right) >\varepsilon _{\alpha
}$ if $w$ does not represent the identity of $G$ and $d_{\alpha
}\left( \varphi \left( w\right) ,e_{\alpha }\right) <1/n$ otherwise. One can
choose for any element $g\in B_{G,S}\left( m\right) $ a word $w_{g}\in F_{X}$
of minimal length that represents $g$ in such a way that $%
w_{g^{-1}}=w_{g}^{-1}$ for $g\in G$. Define then $\pi \left( g\right)
=\varphi \left( w_{g}\right) $ for every $g\in B_{G,S}\left( m\right) $ and
arbitrarily for $g\notin B_{G,S}\left( m\right) $. We claim that $\pi $
is an $\left( m,\varepsilon _{\alpha }\right) $-approximation for $G$.
Indeed, suppose that $g,h\in B_{G,S}\left( m\right) $ are such that $gh\in
B_{G,S}\left( m\right) $. Observe that the word $w_{g}w_{h}w_{gh}^{-1}\in
F_{X}$ has length at most $n$ and represents the identity of $G$.
Hence, by assumption, we have that%
\begin{equation*}
d_{\alpha }\left( \pi \left( g\right) \pi \left( h\right) ,\pi
\left( gh\right) \right) =d_{\alpha }(\varphi (w_{g}w_{h}w_{gh}^{-1}),e_{\alpha
})<1/n< 1/m\text{.}
\end{equation*}%
Suppose now that $g,h$ are distinct elements of $B_{G,S}\left( m\right) $.
We have that $w_{g}w_{h}^{-1}$ is an element of $F_{X}$ of length at most $2m<n$
that does not represent the identity of $G$. Hence%
\begin{equation*}
d_{\alpha }\left( \pi \left( g\right) ,\pi \left( h\right) \right)
=d_{\alpha }\left( \varphi \left( w_{g}w_{h}^{-1}\right), e_{\alpha }\right)
>\varepsilon_\alpha> \varepsilon_\alpha-1/m\text{.}
\end{equation*}%
(3): Fix $m\in \mathbb{N}$. Set $n=\max \left\{ Dehn_{G,S}\left( m\right)
m,N_{G,S}\left( m\right) \right\} $. If $k=\mathcal{R}_{G,S}^{\mathcal{F}}\left(
n\right) $, then for some $\alpha $ with $k_{\alpha }=k$ there exists a
homomomorphism $\varphi\colon F_{X}\rightarrow G_{\alpha }$ such that for any
word $w\in F_{X}$ of length at most $n$ that does not represent the identity
of $G$ one has that $d_{\alpha }\left( \varphi \left( w\right) ,e_{\alpha
}\right) >\varepsilon _{\alpha }$, and $d_{\alpha }\left( \varphi \left(
r\right) ,e_{\alpha }\right) <1/n$ for any relator $r\in R$ of length at most $n$.
Suppose now that $w$ is a word in $G$ of length at most $m$ that represents
the identity of $G$. Then we can write $w$ as the product $\left( \eta
_{1}r_{1}\eta _{1}^{-1}\right) \cdots \left( \eta _{\ell }r_{\ell }\eta
_{\ell }^{-1}\right) $ where $\ell \leqslant Dehn_{G,S}\left( m\right) $, $%
r_{i}\in R$ have length at most $N_{G,S}\left( m\right) \leqslant n$ and $\eta
_{i}\in F_{X}$. Thus, we have, for each $i\leqslant \ell $,%
\begin{equation*}
d_{\alpha }\left( \varphi \left( \eta _{i}r_{i}\eta _{i}^{-1}\right)
,e_{\alpha }\right) =d_{\alpha }\left( \varphi \left( r_{i}\right)
,e_{\alpha }\right) <1/n
\end{equation*}%
and, hence, $d_{\alpha }\left( \varphi \left( w\right) ,e_{\alpha }\right)
<Dehn_{G,S}\left( m\right) /n\leqslant 1/m$.
\end{proof}

\begin{proposition}\label{Proposition:bounded}
Let $\mathcal{F}=\left(
G_{\alpha }, d_{\alpha}, k_{\alpha },\varepsilon _{\alpha }\right) _{\alpha \in I}$ be a family as above such that
\begin{itemize}
\item[(i)] for every $k\in 
\mathbb{N}$, $\left\{ \alpha \in I:k_{\alpha }\leqslant k\right\} $ is finite, and
\item[(ii)] for every $\alpha \in I$ the bi-invariant metric group $\left( G_{\alpha
}, d_{\alpha }\right) $ is compact.
\end{itemize}

 If $G$ is a finitely generated group, then the
following assertions are equivalent:

\begin{enumerate}
\item the $\mathcal{F}$-profile $\mathcal{D}_{G,S}^{\mathcal{F}}$ of $G$ is bounded
(equivalently, each of the profiles
$\mathcal{W}_{G,S}^{\mathcal{F}},\mathcal{R}_{G,S}^{\mathcal{F}}$ is bounded);

\item there exists an injective group homomorphism $G\hookrightarrow G_{\alpha }$
for some $\alpha \in I$.
\end{enumerate}
\end{proposition}

\begin{proof}
The implication (2)$\Rightarrow $(1) is obvious, so we focus on the
implication (1)$\Rightarrow $(2). Fix a finite generating set $S$ for $G$.
Suppose that the $\mathcal{F}$-profile of $G$ is bounded. Then there exists $%
k\in \mathbb{N}$ such that for any $n\in \mathbb{N}$
there exist $\alpha _{n}\in I$ such that $k_{\alpha _{n}}\leqslant k$ and an $%
\left( n,\varepsilon_{\alpha_n}\right) $-approximation $\varphi _{n}\colon G\rightarrow G_{\alpha
_{n}}$. Since by assumption $\left\{ \alpha \in I:k_{\alpha }\leqslant k\right\} $
is finite, without loss of generality we can assume that $\alpha _{n}=\alpha 
$ does not depend on $n$, and thus $\varphi _{n}\colon G\rightarrow G_{\alpha }$
for every $n\in \mathbb{N}$. Fix a nonprincipal ultrafilter $\mathcal{U}$
over $\mathbb{N}$ and define $\varphi\colon G\rightarrow G_{\alpha }$ by $%
\varphi \left( g\right) :=\lim_{n\rightarrow \mathcal{U}}\varphi _{n}\left(
g\right) $. Observe that this is well defined since, by assumption, $\left(
G_{\alpha },d_{\alpha }\right) $ is compact. Since, for every $n\in \mathbb{N%
}$, $\varphi _{n}$ is an $\left( n,\varepsilon_\alpha\right) $-approximation, it follows
that $\varphi $ is an injective group homomorphism.
\end{proof}

\begin{remark}
It is clear from the preceding proof that assumption (ii) can be weakened. For example, Proposition~\ref{Proposition:bounded}
remains true under condition (ii') that $(G_\alpha, d_\alpha)$ is a proper metric space (i.e. a metric space where every closed ball is compact)
and for each $g\in G$ and some $l>0$ we have $\left\{ n\in \mathbb N  \mid d_\alpha(\varphi_n(g), e_\alpha) \leqslant l\right\}\in \mathcal U$, where 
$\alpha$ is the index appearing in the proof. Indeed, such (ii') ensures that $\lim_{n\rightarrow \mathcal{U}}\varphi _{n}\left(g\right)$ exists and is unique.
\end{remark}

\subsection{$\mathcal{F}$-approximations and metric ultraproducts}\label{sec:multra}

Adopting the notation from the beginning of the section, we let $\mathcal{F}$
be the family $\left( G_{\alpha }, d_{\alpha }, k_{\alpha }, \varepsilon _{\alpha }\right)
_{\alpha \in I}$, where $k_{\alpha }$ is a nonzero natural number, $
\varepsilon _{\alpha }$ is a strictly positive real number such that $\inf_{\alpha }\varepsilon _{\alpha }>0$, and $G_{\alpha }$
is a bi-invariant metric group with distance $d_{\alpha }$ and identity
element $e_{\alpha }$. Fix a non-principal ultrafilter $\mathcal{U}$ over the index set $
I $. The \emph{metric ultraproduct} $\prod_{\mathcal{U}}\left( G_{\alpha
},d_{\alpha }\right) $ of the family of bi-invariant metric groups $\left(
G_{\alpha },d_{\alpha }\right) _{\alpha \in I}$ can be defined as in \cite[Section 4]{pestov_introduction_2012}. This is the quotient of the direct
product $\prod_{\alpha \in I}G_{\alpha}$ by the normal subgroup consisting of
those elements $\left( g_{\alpha }\right) $ such that $\lim_{\alpha
\rightarrow \mathcal{U}}d_{\alpha }\left( g_{\alpha },e_{\alpha }\right) =0$
. The metric ultraproduct $\prod_{\mathcal{U}}\left( G_{\alpha },d_{\alpha }\right) 
$ is endowed with a canonical bi-invariant metric $d_{\mathcal{U}}$,
obtained as the quotient of the bi-invariant pseudometric on $\prod_{\alpha
\in I}G_{\alpha }$ defined by $d_{\mathcal{U}}\left( \left( g_{\alpha
}\right) ,\left( h_{\alpha }\right) \right) =\lim_{\alpha \rightarrow 
\mathcal{U}}d_{\alpha }\left( g_{\alpha },h_{\alpha }\right) $. Such a construction is
in fact a particular instance, in the case of bi-invariant metric groups, of
the notion of ultraproduct in the logic for metric structures; see \cite[Section 5]{ben_yaacov_model_2008} and \cite[Section 2.6]
{capraro_introduction_2015}.

In the case when the family $\left( \varepsilon _{\alpha }\right)_{\alpha \in I} $ is
constantly equal to a given strictly positive real number $\varepsilon $,
one can reformulate the notion of $\mathcal{F}$-approximability in terms of
embeddings into a metric ultraproduct $\prod_{\mathcal{U}}\left( G_{\alpha
},d_{\alpha }\right) $. Precisely, a countable group $G$
is $\mathcal{F}$-approximable if and only if there exist an ultrafilter $
\mathcal{U}$ over $I$ and a group homomorphism $\iota\colon G\rightarrow \prod_{
\mathcal{U}}\left( G_{\alpha },d_{\alpha }\right) $ such that $d_{\mathcal{U}
}\left( \iota \left( g\right) ,\iota \left( h\right) \right) \geqslant \varepsilon $
for every nontrivial distinct $g,h\in G$; see also \cite[Proposition 1.9]{thom_about_2012}.

\begin{example}[Varieties and non-varieties]\label{ex:variety}
It follows from the preceding that if $\mathcal{C}$ is a class of groups that is
closed under taking arbitrary direct products, subgroups, and quotients (equivalently, by the Birkhoff theorem, if $\mathcal{C}$  is a \emph{variety of groups}, i.e. a 
class of groups defined by a given set of \emph{identities} or, using an alternative terminology, \emph{laws}), and that contains $%
G_{\alpha }$ for every $\alpha \in I$, then $\mathcal{C}$ contains every $%
\mathcal{F}$-approximable group. Conversely, if all $\mathcal{F}$-approximable groups belong to $\mathcal F$ and $\mathcal F$ is closed
under taking quotients, then the family $\mathcal F$  is a variety of groups.

For instance, for any positive integer $\ell$,  groups approximable by solvable groups of derived length at most $\ell$ are
solvable of derived length at most $\ell$. In contrast, there exists a non-solvable group which is approximable by solvable groups
(with no uniform bound on the derived length). These are  not using any consideration about a metric\footnote{In fact, the metric is the trivial $\{0,1\}$-valued metric $d_{\{0,1\}}$, see Example~\ref{ex:constant}.}.

A more `metric' example is  given be the family $\mathcal F^{fin}_{cc}=(H_\alpha, d^{cc}_{\alpha}, \vert H_\alpha\vert, 1)_{\alpha\in I},$ see Section~\ref{sec:ws},
where $H_\alpha$, $\alpha\in I$ are all finite groups. By our remark, there exists an $\mathcal F^{fin}_{cc}$-approximable group which does not belong to $\mathcal F^{fin}_{cc}$ but, moreover,
there exists a finitely presented group, the famous Higman group, which is not $\mathcal F^{fin}_{cc}$-approximable~\cite{thom_about_2012}, cf. Question~\ref{q:solv}.
\end{example}

\begin{example}[Constant dimension]\label{ex:constant}
Assume that $k_\alpha=k$ for all $\alpha\in I$ and a fixed $k>0.$ By Definition~\ref{def:profile}, the $\mathcal F$-profile
$\mathcal{D}_{G,S}^{\mathcal{F}}$ and, by Proposition~\ref{Prop:DWR}, profiles
$\mathcal{W}_{G,S}^{\mathcal{F}},\mathcal{R}_{G,S}^{\mathcal{F}}$
are constant. In particular, (1) of Proposition~\ref{Proposition:bounded}  holds. 
However, (i) of Proposition~\ref{Proposition:bounded} is not fulfilled.
It is natural to ask whether or not conclusion (2) of Proposition~\ref{Proposition:bounded} remains true 
for a finitely generated $\mathcal F$-approximable group $G$
in this case of constant dimension of approximating groups.

For instance, if $G_\alpha$ are finite for all $\alpha\in I$ and $k_\alpha=|G_\alpha|=k$ for some $k>0$, is the fixed cardinality of these finite groups\footnote{There are finitely many finite groups of given cardinality, that is, at least one group among $G_\alpha, \alpha\in I$ appears infinitely many times in the metric ultraproduct $\prod_{\mathcal{U}}\left( G_{\alpha
},d_{\alpha }\right) $.}, then
Definition~\ref{def:app}(2) implies that $G$ is finite and injects in some $G_\alpha$ (cf. Section~\ref{sec:ws}, for a general setting when sizes of finite groups vary).
Thus, Proposition~\ref{Proposition:bounded}(2) does hold in this case.

In the same vein, let $G_\alpha=GL(k, \mathbb K_\alpha)$ be the group of $k\times k$ invertible matrices for a fixed $k>0$, with coefficients in a field $\mathbb K_\alpha$, and equipped with the \emph{trivial $\{0,1\}$-valued metric $d_{\{0,1\}}$}, for $\alpha\in I$, defined by $d_{\{0,1\}}(g,h)=1$ if $g\not=h$ and $0$ otherwise. Then $G$ is linear, i.e. $G$ is a subgroup of $GL(n, \mathbb K)$ for some positive integer $n$ and a field $\mathbb K$ (which is not required to coincide with  $\mathbb K_\alpha$ for some $\alpha\in I$).
Indeed,  $n$ can be chosen to be $k$ and $\mathbb K$ to be the \emph{algebraic} ultraproduct of fields $\mathbb K_\alpha$.  
This result is due to Malcev~\cite{malcev_linear} and it can be also formulated using the language of the  first-order classical logic:
if the universal first-order theory of $G$ contains that of a linear group $GL(k, \mathbb K_\alpha)$ (this is the case, for instance, when 
$G$ and $GL(k, \mathbb K_\alpha)$ have the same elementary theories), then $G$ is linear.

Note that we can relax the assumption on finite generation of $G$ as by Malcev's local theorem a group has a faithful linear representation of degree $n$ over a field 
of characteristic $p\geqslant 0$
if and only if each of its finitely generated subgroups has such a representation over a field of characteristic $p$~\cite{malcev_linear}.

Since we deal with \emph{metric} ultraproducts (associated to arbitrary bi-invariant distances $d_\alpha$) it is natural to ask
whether or not Malcev's result still holds in the continuous logic setting, see Question~\ref{q:cmalcev}.

However,  it is clear that for an arbitrary group $G_\alpha$ one cannot expect that an embedding of $G$ into an ultrapower of $G_\alpha$
induces an embedding $G\hookrightarrow G_\alpha$. Here is a concrete example in the case of algebraic ultrapower:
if $G_\alpha=F_X$ is a free non-abelian group, then, by a result of Remeslennikov~\cite{rem}, a finitely generated subgroup $G$ of an algebraic
ultrapower of $F_X$ is a fully residually free group. Take such a non-free group $G$, i.e. any non-free Sela's limit group. See also Question~\ref{q:upower}.
\end{example}

\section{Classes of $\mathcal{F}$-approximable groups}

The notion of $\mathcal{F}$-approximable group for various choices of $%
\mathcal{F}$ captures many classes of groups that are currently a subject of intensive study. 
We mention only main examples below, for more details about these wide classes of groups and a broad range of applications see ~\cite{pestov_hyperlinear_2008,pestov_introduction_2012,ceccherini-silberstein_cellular_2010, arzhantseva_approx_2012, arzhantseva_linear_2012, capraro_introduction_2015}.
The reader is also invited to analyze the $\mathcal{F}$-profile and $\mathcal F$-dimension,
given her/his favorite family $\mathcal F.$ Moreover, note that our quantifying of metric approximations extends immediately to quantifying of constraint metric approximations~\cite{arzhantseva_centr_2016}, using suitable \emph{constraint} metric profiles.

Observe that the classes of groups below are so that
all residually finite and all amenable groups  belong to these classes.

\subsection{Sofic groups}
This class of groups has been first introduced by Gromov in~\cite{gromov_endomorphisms_1999} in the context of symbolic dynamics; see also a work of Weiss~\cite{weiss_sofic_2000}.

Let $\mathcal{F}^{sof}$ be the collection of permutation
groups $\Sym(n), n\in \mathbb N$, endowed with the normalized \emph{Hamming distance}:
for permutations $\sigma, \tau\in \Sym(n),$ we define
\begin{equation*}\label{definition_Hamming}
d_{\rm{Ham}}(\sigma,\tau)=\frac{1}{n}\Bigl\lvert\Bigl\{i\mid \sigma(i)\neq\tau(i)\Bigr\}\Bigr\rvert.
\end{equation*}
\begin{definition}[Sofic group via permutations; sofic profile and dimension]
A group $G$ is said to be \emph{sofic} if it is approximable by $\mathcal F^{sof}=(\Sym(n), d_{\rm{Ham}}, n, 1)_{n\in \mathbb N}$, 
in the sense of Definition~\ref{def:app}. 

We call \emph{%
sofic profile }$\mathcal{D}_{G}^{sof}$ and \emph{sofic dimension }$%
\dim_{G}^{sof} $ of a sofic group $G$ the $\mathcal{F}^{sof}$-profile and $\mathcal{F}^{sof}
$-dimension (respectively) for such a choice of the approximating family $\mathcal{F}=\mathcal F^{sof}$.

\end{definition}
We stipulate
that the dimension $k_{n}$ of $\Sym(n)$ is chosen to be equal to $n$ and $\varepsilon
_{n}=1$ for every $n\in \mathbb{N}$. Equivalently, $\varepsilon
_{n}$ can be chosen to be constantly equal to a fixed strictly positive real number $\varepsilon \leqslant 1 $.

The $\simeq $%
-equivalence class of $\mathcal{D}_{G}^{sof}$ and, hence, the value of  $\dim_{G}^{sof}$, do not change if one defines the sofic
profile only considering an $\mathcal{F}^{sof}$-approximation, i.e. a sofic approximation, defined by maps  $g\mapsto \sigma _{g}$ such
that $\sigma _{e_G}$ is the identity permutation and, for $g\neq e_G$, $\sigma
_{g^{-1}}=\sigma _{g}^{-1}$ has no fixed points; see \cite[Exercise 2.1.10]%
{capraro_introduction_2015}. It follows from Proposition \ref{Proposition:bounded} that a finitely-generated group has bounded sofic profile if and only if it is finite.

Gromov's original definition of soficity of a group $G$ 
uses approximations of its Cayley graph  by finite labeled graphs.

\begin{definition}[Sofic group via graphs]
A group $G$ with a finite symmetric generating set $S$ is called \emph{sofic}, if for each $\delta >0$ and each $n\in \mathbb N$ 
there is a finite directed graph $\Gamma$ edge-labeled by $S$, and a subset $\Gamma_0\subseteq \Gamma$ with the properties, that:

\begin{enumerate}
\item[(i)]  For each point $v\in \Gamma_0$ there is a map $\psi_v\colon B_{G,S}\left( n\right)\to \Gamma$ 
which is a label-preserving isomorphism between the ball $B_{G,S}\left( n\right)$ in the Cayley graph of $G$ with respect to $S$ and the $n$-ball in $\Gamma$ around $v$, and
\item[(ii)] $|\Gamma_0|\geqslant (1-\delta)|\Gamma|$.
\end{enumerate}
Such a graph $\Gamma$ is called an \emph{$[n,\delta]$-approximation} of the Cayley graph $Cay(G,S)$.
\end{definition}
We can now give another,  in addition to the above-defined $\mathcal{D}_{G,S}^{sof}$ (and, hence, $\mathcal{W}_{G,S}^{sof}$ and $\mathcal{R}_{G,S}^{sof}),$
natural definition of a profile of a sofic group:

\begin{itemize}
\item $\mathcal{G}_{G,S}^{sof}\left( n\right) $
is the least cardinality (= number of vertices) of  the graph $\Gamma$ in an $[n,1/n]$-approximation of the Cayley graph $Cay(G,S)$.
\end{itemize}

This new profile is $\simeq$-equivalent to our initial definition of the sofic profile.
\begin{proposition}
$\mathcal{G}_{G,S}^{sof}\left( n\right) \simeq \mathcal{D}_{G,S}^{sof}\left( n\right) $.
\end{proposition}
\begin{proof}
The above two definitions of soficity, via permutations and via graphs, are equivalent, see, for instance,
the proof of the equivalence of Definition 4.2 and Definition 4.3 in \cite{elek_sofic_2004}.
Analyzing the details of this proof, we see that 
$\mathcal{G}_{G,S}^{sof}\left( 2n\right) \geqslant \mathcal{D}_{G,S}^{sof}\left( n\right)$ and
$\mathcal{G}_{G,S}^{sof}\left( n\right) \leqslant \mathcal{D}_{G,S}^{sof}\left( 2n+2\right)$, whence the $\simeq$-equivalence of the two functions.
\end{proof}
\begin{example}[Sofic profiles of free abelian groups]\label{ex:ab}
We begin with the group of integers: $G=\mathbb{Z}$ and $S=\left\{ +1,-1\right\} $. Clearly, $\mathcal{D%
}_{\mathbb{Z},S}^{sof}\left( n\right) \leqslant 2n+1$ by considering the sofic
approximation coming from the  action $\upsilon$ of $\mathbb{Z}$ on $\mathbb{Z}/(2n+1)\mathbb{Z} $ 
defined by 
$$
\begin{array}{ccccc}
\upsilon(i) & : & \mathbb{Z}/(2n+1)\mathbb{Z} & \rightarrow &  \mathbb{Z}/(2n+1)\mathbb{Z}\\
 &  & \overline{j} &\mapsto & \upsilon(i)(\overline{j} ) = \overline{i+j}, 
\end{array}
$$
where $i,j\in \mathbb Z$ and $ j\mapsto \overline{j}$ is the canonical epimorphism $\mathbb Z \twoheadrightarrow \mathbb{Z}/(2n+1)\mathbb{Z}$. This shows that $\mathcal{D}_{\mathbb{Z}}^{sof}\left( n\right)
\preccurlyeq n$. 

Let us check that $\mathcal{D}_{\mathbb{Z}}^{sof}\left(
n\right) \simeq n$. Suppose that $k< 2n+1$ and assume that $%
\varphi\colon \mathbb{Z}\rightarrow \Sym(k)$ is an $\left( 2n,1\right) $%
-approximation of $\mathbb{Z}$. Then, since $d_{\rm{Ham}}$ has values in $\{0, 2/k, \ldots, (k-1)/k, 1\}$, we have that for every $i,j\in \left[
-n,n\right] $, $\varphi \left( i+j\right) =\varphi \left( i\right) \varphi
\left( j\right) $ and $\varphi \left( i\right) \ $is a nontrivial element of 
$\Sym(k)$ whenever $i\not=0$. In particular, $\varphi \left( i\right) = \varphi \left( 1\right)^i$ and $\varphi \left( 1\right) $ is an element of order $>n$. 
We stipulate that $\varepsilon = 1$, then for every $i\in [-n,n]$ such that $i\not=1$ we have
$$
d_{\rm{Ham}}(\varphi \left( 1\right),\varphi \left( 1\right)^i) >1-\frac{1}{2n}.
$$
This implies  $d_{\rm{Ham}}(\varphi \left( 1\right),\varphi \left( 1\right)^i)=1$.
Therefore, $\varphi \left( 1\right)$ is a cycle of length $>n$.
Hence, $k> n$. This shows that $n\preccurlyeq \mathcal{D}_{%
\mathbb{Z}}^{sof}\left( n\right) $. It follows from this and Proposition \ref%
{Proposition:finite-index} that any virtually-$\mathbb{Z}$ group $G$ has
 $\mathcal{D}_{G}^{sof}\left( n\right) \simeq n$. It also follows
from this and the estimate on the sofic profile of the direct product that $%
\mathcal{D}_{\mathbb{Z}^{d}}^{sof}\left( n\right) \preccurlyeq n^{d}$ for every $%
d\in \mathbb{N}$, see Section  \ref{sec:direct}. Moreover, $\mathcal{D}_{\mathbb{Z}%
^{d}}^{sof}\left( n\right) \simeq n^{d}$ and  $\mathcal{D}_{G}^{sof}\left( n\right) \simeq n^{d}$ for any virtually-$\mathbb{Z}^d$ group $G$. This follows from the same argument as above combined with a result
on the stability of the commutator relator words with respect to the Hamming distance 
\cite[Corollary 6.5]{arzhantseva_almost_2015}, see Corollary \ref{cor:stab}.
\end{example}

\subsection{Hyperlinear groups}\label{sec:hyp}
This class of groups appeared in relation to the concept of hyperlinearity in operator algebras.
The definition below is due to a result of R\u{a}dulescu \cite{radulescu_von_2008}; see also \cite[Proposition 2.2.9]%
{capraro_introduction_2015} and \cite[Section 4.2]{hayes_metric_2018}.

Let $\mathcal{F}^{hyp}$ be the collection of finite-rank unitary groups $U(n), n\in \mathbb N,$
endowed with the normalized \emph{Hilbert-Schmidt distance}:
for unitary matrices $u=(u_{ij}), v=(v_{ij}) \in U(n),$ we define
\begin{equation*}\label{definition_HS}
d_{\rm{HS}}(u,v)=\| u-v\|_2=\sqrt{\frac{1}{n}\sum_{i,j=1}^n\lvert u_{ij}-v_{ij}\rvert^2}=\sqrt{\frac{1}{n}{\rm tr}((u-v)^\ast(u-v))}.
\end{equation*}
\begin{definition}[Hyperlinear group; hyperlinear profile and dimension]\label{def:hyp}
A group $G$ is said to be \emph{hyperlinear} if it is approximable by $\mathcal F^{hyp}=(U(n), d_{\rm{HS}}, n, \sqrt{2})_{n\in \mathbb N}$, 
in the sense of Definition~\ref{def:app}. 

We call \emph{%
hyperlinear profile }$\mathcal{D}_{G}^{hyp}$ and \emph{hyperlinear dimension }$%
\dim_{G}^{hyp} $ of a hyperlinear group $G$ the $\mathcal{F}^{hyp}$-profile and $\mathcal{F}^{hyp}
$-dimension (respectively) for such a choice of the approximating family $\mathcal{F}=\mathcal F^{hyp}$. 
\end{definition}

Again we convene that
the dimension $k_{n}$ of $U(n)$ is equal to $n$ and  $\varepsilon
_{n}=\sqrt{2}$ for every $n\in \mathbb{N}$. Equivalently, $\varepsilon
_{n}$ can be chosen to be constantly equal to a fixed strictly positive real number $\varepsilon \leqslant \sqrt{2}$.

The flexibility in the choice of $\varepsilon
_{n}$ that we observe in the definitions of both sofic and hyperlinear groups does not hold a priori for arbitrary $\mathcal F$-approximations. Indeed, 
it strongly depends on specific properties of the metric we use. In fact, both the Hamming and the Hilbert-Schmidt metrics
behave well under the so-called ``amplification trick'', see the discussions in~\cite{pestov_hyperlinear_2008} and~\cite{arzhantseva_linear_2012}, whence
this freedom in the choice of $\varepsilon_n$'s in the definitions of sofic and hyperlinear groups, respectively.

It follows from Proposition \ref{Proposition:bounded} that a finitely-generated group has bounded hyperlinear profile if and only if it embeds into $U(n)$ for some $n\in \mathbb{N}$.

Given two permutations $\sigma, \tau\in \Sym (n)$, let $u_\sigma, v_\tau\in U(n)$ denote the corresponding
permutation matrices. Then,
$$
d_{\rm{Ham}}(\sigma,\tau)=\frac12(d_{\rm{HS}}(u_\sigma,v_\tau))^2.
$$
It follows that sofic groups are hyperlinear.
Noting $B_{G,S}(n)\subseteq B_{G,S}(2n^2)$,  we immediately obtain that, for a sofic group $G$, we have:
$$
\mathcal{D}_{G,S}^{hyp}(n)\leqslant \mathcal{D}_{G,S}^{sof} (2n^2).
$$
The converse is not yet known: whether or not all hyperlinear groups are sofic is a well-known open problem.
Observe that the Hamming distance is an $\ell^1$-type metric and the Hilbert-Schmidt distance is the Euclidean, hence, an $\ell^2$-type metric.
Therefore, the above square root distortion of the distance under the canonical map $\Sym(n)\hookrightarrow U(n)$ sending permutations to the permutation matrices: $ \sigma\mapsto u_\sigma, \tau\mapsto v_\tau,$  can a priori not be improved into an isometric embedding.
However, we deal with approximations, whence the following
\begin{conjecture}\label{conj:sofhyp}
If $G$ is sofic, then $\mathcal{D}_{G}^{sof} (n)\preccurlyeq \mathcal{D}_{G}^{hyp}(n)$.
\end{conjecture}
The conjecture holds for sofic $\mathcal F^{hyp}$-stable groups (see Definition~\ref{def:stab}), e.g. for virtually abelian groups and the Heisenberg groups, see~Corollary~\ref{cor:stab}\ref{cor:stab2} and Example~\ref{ex:stablegps}. See also Question~\ref{q:sofhyp} and Question~\ref{q:sofhyps}.

Here is a useful modification of the hyperlinear profile.
The finite-rank unitary groups $U(n), n\in \mathbb N,$ can be endowed with the normalized \emph{%
projective Hilbert-Schmidt pseudo-distance}:  for unitary matrices $u, v\in U(n),$ we define%
\begin{equation*}
d_{\overline{\mathrm{HS}}}(u,v)=\inf_{\lambda \in \mathbb{T}}\sqrt{\frac{1}{n}%
\mathrm{tr}\left((u-\lambda v)^{\ast }(u-\lambda v)\right),}
\end{equation*}%
where $\mathbb{T}$ denotes the set of complex numbers of modulus $1$.

\begin{definition}[Projective hyperlinear profile and dimension]\label{def:phyp}
Let $G$ be a finitely generated group with a finite
generating set $S$. The \emph{projective hyperlinear profile} of $G$ is the function $%
\mathcal{D}_{G,S}^{\overline{hyp}}\colon\mathbb{N}\rightarrow \mathbb{N}\cup \left\{
+\infty \right\} $ defined by setting $\mathcal{%
D}_{G,S}^{\overline{hyp}}\left( n\right) $ to be the least $k$ such that there exists
a function $\sigma\colon G\rightarrow U\left( k\right) $ such that

\begin{itemize}
\item[(1)] $d_{\mathrm{HS}}\left( \sigma \left( g\right) \sigma \left( h\right)
,\sigma \left( gh\right) \right) <1/n$ for every $g,h\in B_{G,S}\left( n\right) $ with $gh\in B_{G,S}\left( n\right)$, and

\item[(2)] $d_{\overline{\mathrm{HS}}}\left( \sigma \left( g\right) ,\sigma \left(
h\right) \right) >\sqrt{2}-1/n$ for every $g,h\in B_{G,S}\left( n\right) $
such that $g\neq h$.
\end{itemize}

The \emph{projective hyperlinear dimension} of $G$
is defined by
$$
\dim_{G,S}^{\overline{hyp}}=\limsup_{n\rightarrow +\infty}\frac{1}{n}\log \mathcal{D}%
_{G,S}^{\overline{hyp}}\left( n\right). 
$$
\end{definition}

Two distinct (pseudo)distances on $U(k)$ are used in the two conditions above.
One might consider the existence of a map $\sigma \colon G\rightarrow U\left( k\right)$ from Definition~\ref{def:phyp}
as an alternative definition of a hyperlinear group. Indeed, a result of R\u{a}dulescu~\cite{radulescu_von_2008} is that this
is actually equivalent to Definition~\ref{def:hyp}. Our next result shows a precise
relationship between these two approaches on the quantifying level.

For the proof of the following proposition, observe that $d_{\mathrm{HS}%
}\left( u,e_{U(n)}\right) ^{2}=2-\frac{2}{n}\mathrm{Re}\left( \mathrm{tr}\left(
u\right) \right) $, and $d_{\overline{\mathrm{HS}}}\left( u,e_{U(n)}\right) ^{2}=2-\frac{2%
}{n}\left\vert \mathrm{tr}\left( u\right) \right\vert $ for $u\in U\left(
n\right) $. Observe furthermore that $d_{\overline{\mathrm{HS}}}\left( u,v\right)
\leqslant d_{\mathrm{HS}}\left( u,v\right) $ for $u,v\in U\left( n\right) $. Our next result is based on the \textquotedblleft
amplification trick\textquotedblright ; see \cite{radulescu_von_2008}, \cite[%
Proposition 2.2.9]{capraro_introduction_2015}, \cite[Section 4.2]%
{hayes_metric_2018}.

\begin{proposition}
\label{Proposition:projective-hyperlinear estimate} Let $G\ $ be a
group with a finite generating set $S$. Then 
\begin{equation*}
\mathcal{D}_{G,S}^{hyp}\left( n\right) \leqslant \mathcal{D}_{G,S}^{\overline{hyp}}\left(
n\right)
\end{equation*}
and, for $n\geqslant \left( \frac{4\sqrt{2}}{5}-1\right) ^{-1}$, 
\begin{equation*}
\mathcal{D}_{G,S}^{\overline{hyp}}\left( n\right) \leqslant \left(2\mathcal{D}_{G,S}^{hyp}%
\left( 40n\right) \right)^{\ell }
\end{equation*}
where $\ell =\left\lceil \log \left( \frac{1}{\frac{\sqrt{2}}{20n}-\frac{1}{%
200n^{2}}}\right) \log \left( \frac{5}{4}\right) ^{-1}\right\rceil $.
\end{proposition}

\begin{proof}
The first inequality is obvious. Let us check the second inequality. For $%
u\in U\left( k\right) $,  we set $\tau \left( u\right) =\frac{1}{k}\mathrm{tr}\left(
u\right) $. Suppose that $n\geqslant \left( \frac{4\sqrt{2}}{5}-1\right) ^{-1}$.
Suppose that $\sigma\colon G\rightarrow U\left( k\right) $ is an $(40n,\sqrt{2})$%
-approximation for $\left( G,S\right) $. Then for $g\in B_{G,S}\left(
n\right) $ such that $g\neq e_G$ we have that $d_{\mathrm{HS}}\left( \sigma
\left( g\right) ,e_{U(k)}\right) \geqslant \sqrt{2}-\frac{1}{40n}$. Therefore, 
\begin{equation*}
2-2\mathrm{Re}\left( \tau \left( \sigma \left( g\right) \right) \right) =d_{%
\mathrm{HS}}\left( \sigma \left( g\right) ,e_{U(k)}\right) ^{2}\geqslant 2-\frac{1}{10n}%
\text{.}
\end{equation*}%
Hence $\mathrm{Re}\left( \tau \left( u\right) \right) \leqslant \frac{1}{20n}$.
Consider the map $\widetilde{\sigma }\colon G\rightarrow U\left( 2k\right) $
defined by 
\begin{equation*}
\widetilde{\sigma }\left( g\right) =%
\begin{bmatrix}
\sigma \left( g\right) & 0 \\ 
0 & 1%
\end{bmatrix}%
\text{.}
\end{equation*}%
Then for $g\in G$ one has that, since $n\geqslant \left( \frac{4\sqrt{2}}{5}%
-1\right) ^{-1}$, 
\begin{equation*}
\left\vert \tau \left( \widetilde{\sigma}(g)\right) \right\vert =\frac{\left\vert 1+\tau \left(
\sigma(g)\right) \right\vert }{2}\leqslant \frac{\sqrt{2}+\frac{1}{20n}}{2}\leqslant \frac{1+%
\frac{1}{20n}}{\sqrt{2}}\leqslant \frac{4}{5}\text{.}
\end{equation*}

Set $\delta :=\frac{\sqrt{2}}{20n}-\frac{1}{200n^{2}}$. Fix $\ell \in 
\mathbb{N}$ such that 
\begin{equation*}
\log \left( \frac{5}{4}\right) \ell \geqslant \log \left( \frac{1}{\delta }%
\right) \text{.}
\end{equation*}%
Consider the map $\widetilde{\sigma }^{\otimes \ell }\colon G\rightarrow U\left(
(2k)^{\ell }\right) $ defined by $\widetilde{\sigma }^{\otimes \ell }\left(
g\right) =\widetilde{\sigma }\left( g\right) \otimes \widetilde{\sigma }%
\left( g\right) \otimes \cdots \otimes \widetilde{\sigma }\left( g\right) $.
Then we have that, for $g\in B_{G,S}\left( n\right) $ such that $g\neq e_G$,
since 
\begin{equation*}
\left\vert \tau (\widetilde{\sigma }^{\otimes \ell }\left( g\right)
)\right\vert =\left\vert \tau \left( \widetilde{\sigma }\left( g\right)
\right) \right\vert ^{\ell }\leqslant \left( \frac{4}{5}\right) ^{\ell }\leqslant
\delta \text{.}
\end{equation*}%
Therefore, we have that%
\begin{equation*}
d_{\overline{\mathrm{HS}}}(\widetilde{\sigma }^{\otimes \ell }\left( g\right)
,e_{U((2k)^{\ell})})^{2}=2-2\left\vert \tau \left( \widetilde{\sigma }^{\otimes \ell }\left(
g\right) \right) \right\vert \geqslant 2-2\delta \text{.}
\end{equation*}%
Hence,%
\begin{equation*}
d_{\overline{\mathrm{HS}}}(\widetilde{\sigma }^{\otimes \ell }\left( g\right)
,e_{U((2k)^{\ell})})\geqslant \sqrt{2-2\delta }> \sqrt{2}-1/10n\text{.}
\end{equation*}%
Observe now that for $g,h\in B_{G,S}\left( n\right) $ with $gh\in B_{G,S}\left( n\right)$ one has that%
\begin{equation*}
d_{\mathrm{HS}}(\widetilde{\sigma }^{\otimes \ell }\left( g\right) 
\widetilde{\sigma }^{\otimes \ell }\left( h\right) ,\widetilde{\sigma }%
^{\otimes \ell }\left( gh\right) )\leqslant d_{\mathrm{HS}}(\sigma \left(
g\right) \sigma \left( h\right) ,\sigma \left( gh\right) )\leqslant 1/40n\text{.}
\end{equation*}%
Finally, if $g,h\in B_{G,S}\left( n\right) $ are such that $g\neq h,$ then, using the bi-invariance of $d_{\overline{\mathrm{HS}}}$, the triangle inequality, 
and the almost homomorphism condition on $d_{HS}$,  and hence on $d_{\overline{HS}}$, (cf. Lemma~\ref{Lemma:approximate-homomorphism} (5)) we
have that%
\begin{eqnarray*}
d_{\overline{\mathrm{HS}}}(\widetilde{\sigma }^{\otimes \ell }\left( h\right) ,%
\widetilde{\sigma }^{\otimes \ell }\left( g\right) ) &\geqslant &d_{\overline{HS}}(\widetilde{%
\sigma }^{\otimes \ell }\left( g^{-1}h\right) ,e_{U((2k)^{\ell})})-\frac{1}{8n} \\
&\geqslant &\sqrt{2}-1/10n-1/8n\geqslant \sqrt{2}-1/n\text{.}
\end{eqnarray*}%
This concludes the proof.
\end{proof}

\subsection{Linear sofic groups}\label{sec:linsof}
This class of groups has been introduced and systematically studied by Arzhantseva-P\u{a}unescu~\cite{arzhantseva_linear_2012}.
The next definition is due to \cite[Proposition 4.4]{arzhantseva_linear_2012}.

Let $\mathcal{F}^{lin}$ be the collection of groups $GL(n, \mathbb{K})$ of $n\times n$ invertible matrices with coefficients in a given filed $\mathbb K,$ 
endowed with the normalized \emph{rank distance}:
for invertible matrices $u, v \in GL(n, \mathbb{K}),$ we define
\begin{equation*}\label{definition_HS}
d_{\rm{rank}}(u,v)=\frac{1}{n}{\rm rank}( u-v).
\end{equation*}
\begin{definition}[Linear sofic group; linear sofic profile and dimension]\label{def:lsof}
A group $G$ is said to be \emph{linear sofic over a field $\mathbb K$} if it is approximable by 
$\mathcal F^{lin}=(GL(n, \mathbb{K}), d_{\rm{rank}}, n, 1/4)_{n\in \mathbb N}$, 
in the sense of Definition~\ref{def:app}.

 We call \emph{%
linear sofic  profile }$\mathcal{D}_{G}^{lin}$ and \emph{linear sofic dimension }$%
\dim_{G}^{lin} $ over a field $\mathbb K$ of a linear sofic group $G$ over a field $\mathbb K$ the $\mathcal{F}^{lin}$-profile and $\mathcal{F}^{lin}
$-dimension (respectively) for such a choice of the approximating family $\mathcal{F}=\mathcal F^{lin}$. 

\end{definition}

Thus, the dimension $k_n$ of $GL\left(n, \mathbb{K}\right) $ is declared to be $n$ and 
 $\varepsilon_{n}$ is constantly equal to  $1/4$. The value of $1/4$ comes from the so-called \emph{rank amplification},
 a construction introduced in~\cite{arzhantseva_linear_2012}.
It is not known whether $1/4$ can be replaced by a larger value (fixed or arbitrarily chosen between $1/4$ and 1), see Question~\ref{q:lsofic}. 

It follows from Proposition \ref{Proposition:bounded} that a finitely-generated group has bounded linear sofic profile if and only if it is linear, namely, if and only if it embeds into  $GL\left(n, \mathbb{K}\right) $ for some $n\in \mathbb N$.

As above, we represent permutations by permutation matrices: $\Sym(n)\ni\sigma\mapsto u_\sigma\in GL\left(n, \mathbb{K}\right)$, for a fixed arbitrary field $\mathbb K$. Observe  that~\cite[Proposition 4.5]{arzhantseva_linear_2012}:
$$
 d_{\rm{rank}}(u_\sigma,  e_{GL\left(n, \mathbb{K}\right)} )\leqslant d_{\rm{Hamm}}(\sigma,  e_{\Sym(n)} )\leqslant 2d_{\rm{rank}}(u_\sigma,  e_{GL\left(n, \mathbb{K}\right)}).
$$
Therefore, sofic groups are linear sofic over any given field $\mathbb K$ and, for a sofic group $G$, we have:
$$
\mathcal{D}_{G,S}^{lin} (n)\leqslant \mathcal{D}_{G,S}^{sof}(n).
$$

Here is a useful modification of the linear sofic profile, which is inspired by the projective variant of the hyperlinear profile. 
The groups $GL\left(n, \mathbb{K}\right), n\in \mathbb N,$ can be endowed with the normalized \emph{Jordan} or \emph{%
projective rank pseudodistance}:  for matrices $u, v\in GL\left(n, \mathbb{K}\right)$, following \cite%
{hayes_metric_2018}, we define%
\begin{equation*}
d_{\overline{\mathrm{rank}}}\left( u,v\right) =\frac{1}{n}\min_{\lambda \in \mathbb{%
F}^{\times }}\mathrm{\mathrm{rank}}\left( u-\lambda v\right),
\end{equation*}%
where $\mathbb{F}$ is the algebraic closure of $\mathbb{K}$ and
the rank is computed in $\mathbb{F}^{n}$.

\begin{definition}[Projective linear sofic profile and dimension]\label{def:plsof}
Let $G$ be a finitely generated group with a finite
generating set $S$. The \emph{projective linear sofic profile} of $G$ is the function $%
\mathcal{D}_{G,S}^{\overline{lin}}\colon\mathbb{N}\rightarrow \mathbb{N}\cup \left\{
+\infty \right\} $ defined by setting $\mathcal{%
D}_{G,S}^{\overline{lin}}\left( n\right) $ to be the least $k$ such that there exists
a function $\sigma\colon G\rightarrow GL\left( k, \mathbb K\right) $ such that

\begin{itemize}
\item[(1)] $d_{\mathrm{rank}}\left( \sigma \left( g\right) \sigma \left( h\right)
,\sigma \left( gh\right) \right) <1/n$ for every $g,h\in B_{G,S}\left(
n\right) $ with $gh\in B_{G,S}\left( n\right)$, and

\item[(2)] $d_{\overline{\mathrm{rank}}}\left( \sigma \left( g\right) ,\sigma \left(
h\right) \right) >\frac{1}{8}-1/n$ for every $g,h\in B_{G,S}\left( n\right) $
such that $g\neq h$.
\end{itemize}

The \emph{projective linear sofic dimension} of $G$
is defined by
$$
\dim_{G,S}^{\overline{lin}}=\limsup_{n\rightarrow +\infty}\frac{1}{n}\log \mathcal{D}%
_{G,S}^{\overline{lin}}\left( n\right). 
$$

\end{definition}

As in the projective hyperlinear case,  two distinct (pseudo)distances on $GL(k, \mathbb K)$ are used in the two conditions above and 
one might consider the existence of a map $\sigma\colon G\to GL(k,\mathbb K)$ from Definition~\ref{def:plsof}  as an alternative definition of a linear sofic group. 
Indeed, a result of Arzhantseva-P\u{a}unescu~\cite[Theorem 5.10]{arzhantseva_linear_2012}  shows that this is actually equivalent to Definition~\ref{def:lsof}. 
Moreover, our next result is that these two approaches are equivalent on the quantifying level as well.

\begin{proposition}
Let $G\ $ be a
group with a finite generating set $S$. Then $\mathcal{D}%
_{G,S}^{lin}\left( n\right) \leqslant \mathcal{D}_{G,S}^{\overline{lin}}\left( n\right) $
and $\mathcal{D}_{G,S}^{\overline{lin}}\left( n\right) \leqslant 2\mathcal{D}%
_{G,S}^{lin}\left( 2n\right) $.
\end{proposition}

\begin{proof}
The first inequality is obvious. The second inequality is proved in \cite[%
Proposition 4.8]{hayes_metric_2018}, using \cite[Theorem 5.10]%
{arzhantseva_linear_2012}.
\end{proof}

We obtain further useful variations of $\mathcal{D}_{G,S}^{{lin}}$ when restricting to other meaningful classes of matrices, still equipped with the normalized rank distance\footnote{Naturally, one can also vary the distance by taking, for example, the normalized operator norm, the Frobenius norm, the $p$-Schatten norm with $1\leqslant p\leqslant\infty$,~etc.}. For example, we use $\mathcal{D}_{G,S}^{{u}}$ when approximate by unitary matrices, $\mathcal{D}_{G,S}^{{sa}}\left( n\right)$ by self-adjoint matrices and $\mathcal{D}_{G,S}^{{nor}}\left( n\right)$ by normal matrices with respect to $d_{\rm{rank}}$. See also Question~\ref{q:srk}.

\subsection{Weakly sofic groups}\label{sec:ws}
This class of groups has been introduced by Glebsky-Rivera~\cite{glebsky_profinite}.

Let $\mathcal{F}^{fin}$ be the collection of all finite groups $H_\alpha, \alpha\in I,$ each of which is
endowed with a normalized bi-invariant distance $d_\alpha$. For example, such a distance can be induced
by the bi-invariant distances on the ambient groups via Cayley's embeddings $H_\alpha\hookrightarrow \Sym(|H_\alpha|)\hookrightarrow U(|H_\alpha|),$
where $\vert H_\alpha\vert$ denotes the cardinality of $H_\alpha$.

\begin{definition}[Weakly sofic group;  weakly sofic profile and dimension]
A group $G$ is said to be \emph{weakly sofic} if it is approximable by 
$\mathcal F^{fin}=(H_\alpha, d_{\alpha}, \vert H_\alpha\vert, 1)_{\alpha\in I}$, 
in the sense of Definition~\ref{def:app}.

We call \emph{%
weakly sofic  profile }$\mathcal{D}_{G}^{fin}$ and \emph{weakly sofic dimension }$%
\dim_{G}^{fin} $ of a weakly sofic group $G$  the $\mathcal{F}^{fin}$-profile and $\mathcal{F}^{fin}
$-dimension (respectively) for such a choice of the approximating family $\mathcal{F}=\mathcal F^{fin}$. 
\end{definition}

Here, the dimension $k_{\alpha }$ is the cardinality $\vert H_\alpha\vert$
of the finite group $H_{\alpha }$ and $\varepsilon _{\alpha }=1$ for every $\alpha \in I$. It follows from Proposition \ref{Proposition:bounded} that a finitely-generated group has bounded weakly sofic profile if and only if it is finite.

Sofic groups are clearly weakly sofic  and, for a sofic group $G$, we have (`$!$' denotes the factorial):
$$
\mathcal{D}_{G, S}^{fin} (n)\leqslant \mathcal{D}_{G, S}^{sof}(n)!
$$
Linear sofic groups are weakly sofic~\cite[Theorem 8.2]{arzhantseva_linear_2012}.
However,  for a linear sofic group $G$, the exact relationship between 
$\mathcal{D}_{G, S}^{fin} (n)$ and $ \mathcal{D}_{G, S}^{lin}(n)$
 yet remains to establish. See Questions~\ref{q:qsofic}, \ref{q:qlsofic1}, and \ref{q:qlsofic2}. 

An interesting subclass of weakly sofic groups has been introduced by Thom~\cite{thom_about_2012}.
Namely, if $\mathcal F^{fin}_{cc}=(H_\alpha, d^{cc}_{\alpha}, \vert H_\alpha\vert, 1)_{\alpha\in I}$ is a family of finite groups where 
each bi-invariant distance $d^{cc}_{\alpha}$ is in addition \emph{commutator-contractive}, then the famous Higman group $H_4$
is not $\mathcal F^{fin}_{cc}$-approximable~\cite{thom_about_2012}, hence, the corresponding metric profile of the Higman group diverges
(i.e. $\mathcal{D}_{H_4,S}^{\mathcal{F}^{fin}_{cc}}\left( n\right)=+\infty$ for a large enough $n\in \mathbb N$).

A great freedom in the choice of bi-invariant distances $d_\alpha$ on finite groups $H_\alpha, \alpha\in I$ suggests that sofic groups might be a proper subset of weakly sofic groups. This is unknown. It is intriguing that the Hamming distance on symmetric groups $\Sym(n), n\in \mathbb N$ plays a distinguished role: the soficity can be defined with no reference to any distance~\cite{glebskyA}, or accurately speaking, with a reference to the trivial $\{0,1\}$-valued distance $d_{\{0,1\}}$ only, see Example~\ref{ex:constant}. See also Question~\ref{q:wsA}.
\subsection{Weakly hyperlinear groups}
This class of groups has been introduced by Gismatullin~\cite{jakub_whyp}.

Let $\mathcal{F}^{ct}$ be the collection of all compact groups $H_\alpha, \alpha\in I,$ each of which is
endowed with a normalized bi-invariant distance $d_\alpha$. Examples of such distances are the trivial $\{0,1\}$-valued distance $d_{\{0,1\}}$ and
the (normalized) \emph{Alexandroff-Urysohn distance}\footnote{In a topological group $H$, 
if $\{U_n\}_{n=0}^{\infty}$ is so that for all $h\in H, n\geqslant 0,$ one has $U_n=U_n^{-1}=hU_nh^{-1}\subseteq H$ and $U_0=H, e_H\in U_n, U_{n+1}^3\subseteq U_n$, then
$d_{AU}(g,h)=\inf\{ l(s_1)+\ldots +l(s_k)\mid s_i\in H, gh^{-1}=s_1\cdots s_k\}$, where $l(s_i)=\inf\{2^{-n} \mid s_i\in U_n\}$, is a bi-invariant distance on $H$. For a compact $H$, this distance is continuous, $(H, d_{AU})$ is of finite diameter, and $d_{AU}$ is compatible with the topology if and only if $H$ is first-countable.}
 $d_{\rm{AU}}$ or, in specific compact groups,
the \emph{conjugacy distance} $d_{\rm{conj}}$ induced by the Haar measure on centerless compact groups\footnote{$d_{\rm{conj}}(g,h)=\frac{\log\mu(C(gh^{-1}))}{\log\mu(H)}$, where $C(gh^{-1})$ denotes the conjugacy class of element $gh^{-1}$ for $g,h\in H$, and $\mu$ is the Haar measure on a compact group $H$ with $\mu(H)\not=1$ and $\mu(C(gh^{-1}))\not=0$.} and
the well-known bi-invariant Riemannian distances on compact Lie groups.

\begin{definition}[Weakly hyperlinear group; weakly hyperlinear profile and dimension]
A group $G$ is called \emph{weakly hyperlinear} if it is approximable by 
$\mathcal F^{ct}=(H_\alpha, d_{\alpha}, {\dim }\, H_\alpha, 1)_{\alpha\in I}$, 
in the sense of Definition~\ref{def:app}.

We call \emph{%
weakly hyperlinear  profile }$\mathcal{D}_{G}^{ct}$ and \emph{weakly hyperlinear dimension }$%
\dim_{G}^{ct} $ of a weakly hyperlinear group $G$  the $\mathcal{F}^{ct}$-profile and $\mathcal{F}^{ct}
$-dimension (respectively) for such a choice of the approximating family $\mathcal{F}=\mathcal F^{ct}$. 
\end{definition}

Here, the dimension $k_{\alpha }={\dim }\, H_\alpha$ can be the Lebesgue covering dimension or cohomological dimension 
of the compact group $H_{\alpha }$, and $\varepsilon _{\alpha }=1$ for every $\alpha \in I$. 
Hyperlinear groups are clearly weakly hyperlinear  and, hence, for a hyperlinear group $G$, we have:
$$
\mathcal{D}_{G, S}^{ct} (n)\leqslant \mathcal{D}_{G, S}^{hyp}(n).
$$

\subsection{LE-$\mathcal F$ groups}\label{sec:lef}
The introduction of these classes of groups goes back to Malcev~\cite{malcev_linear}, see also a work of Vershik-Gordon~\cite{vershik_gordon}.

Let us consider an arbitrary family $\mathcal{F}_{\{0,1\}}$ of discrete groups, where a
discrete group is endowed with the trivial $\left\{
0,1\right\} $-valued metric $d_{\left\{
0,1\right\}}$ (induced by the length function such that the length of each non-trivial element is 1). In this case, the choice of the parameters $%
\varepsilon _{\alpha }$ is of course irrelevant. We call an $\mathcal{F}_{\{0,1\}}$-approximable group simply an \emph{LE-$\mathcal F$-group}.
For example, the famous class of
groups that are \emph{locally embeddable into finite groups }(LEF) coincides
with the class of $\mathcal{F}^{fin}_{\left\{
0,1\right\}}$-approximable groups, where $\mathcal{F}^{fin}_{\left\{
0,1\right\}}=(H_\alpha, d_{\left\{
0,1\right\}}, \vert H_\alpha\vert, 1)_{\alpha\in I}$ is
the collection of all finite discrete groups each of which is endowed with the trivial metric $d_{\left\{
0,1\right\}}$. When $\mathcal{F}^{a}_{\left\{
0,1\right\}}$ is the
family of all amenable groups endowed with the trivial metric $d_{\left\{
0,1\right\}}$, 
one obtains the notion of  an \emph{initially
subamenable} group or, in other terminology, of  a group \emph{locally embeddable into amenable groups } (LEA).
Accordingly, we have the concepts of the \emph{LE-$\mathcal F$ profile} (also called the LE-$\mathcal F$ growth in this case, see Definition~\ref{def:lefgrowth} below) and dimensions. In particular, we have  the LE-$\mathcal F^{fin}$  and LE-$\mathcal F^{a}$ profiles whenever
the dimensions of groups from $\mathcal{F}^{fin}_{\left\{
0,1\right\}}$ and $\mathcal{F}^a_{\left\{
0,1\right\}}$ are chosen. For instance,  for a finite group one can again take its cardinality and for
an amenable group  its asymptotic dimension or its cohomological dimension. Alternatively, generalizing our approach further,
one can use a non-constant `dimension' function of approximating groups. For instance, for approximating amenable groups one can use their F\o lner functions; we formalize this in more detail below
by introducing the concept of the LEA profile, see Section~\ref{sec:lea}.

\section{Relations to other famous quantifying functions}

Many examples of quantifying functions associated with a given finitely generated group
have been investigated in the literature. Since any group is trivially approximated by itself equipped 
with the trivial metric $d_{\left\{
0,1\right\}}$, we can consider such quantifying functions as very specific instances of
our much more general approach.

Residually finite groups and amenable groups (and, hence, residually amenable groups) are basic examples of groups
which are approximable by families we have mentioned in the preceding subsection. Therefore, functions 
quantifying the residual finiteness and amenability can be used to produce
interesting upper bounds to our metric $\mathcal F$-profiles.

Lower bounds are more difficult to provide as our general setting of $\mathcal F$-approximable groups
encompasses many different classes of groups with a priori very distinct quantifying features.

\subsection{Growth of balls and the metric profile}
A famous quantifying function associated to every finitely generated group is the \emph{growth} function
$$\beta_{G,S}(n)=\vert B_{G,S}(n)\vert,$$ the cardinality of the ball at the identity of $G$ with respect to the word length distance induced by the generating set $S$.
This function gives a lower bound for an arbitrary $\mathcal F$-profile. 
Indeed, the uniform injectivity condition, condition (2) of Definition~\ref{def:app},  ensures that the ball $B_{G,S}(n)$ is injected into the corresponding
group $G_\alpha$. It remains to estimate from below the minimal dimension $k_\alpha$ such that $G_\alpha$ can have this  fixed ball injected.
Usually, such an estimate of $k_\alpha$ is immediate (although, in general, it depends on the dimension one has chosen for each $G_\alpha$).
For example,  for a sofic group $G$, we have:
$$
\beta_{G,S}(n)\leqslant \mathcal{D}_{G, S}^{fin} (n)\leqslant \mathcal{D}_{G, S}^{sof}(n)!
$$

\subsection{F\o lner function}
This renown function has been introduced by Vershik~\cite{vershik_folner}. Let $G$ be a group with finite symmetric
generating set $S$.  The \emph{F\text{\o} lner function} of $
G$ with respect to $S$, denoted by $\mathrm{F\text{\o} l}_{G,S}\left( n\right) $,  is defined to be the
smallest size $\vert A\vert$ of a  finite subset $A\subseteq G$ with the property that $$\sum_{g\in
B_{G,S}\left( n\right) }\left\vert gA\bigtriangleup A\right\vert \leqslant
\left\vert A\right\vert /n,$$ 
with $\mathrm{F\o l}_{G,S}\left( n\right) =\infty$ whenever there is no such a subset $A\subseteq G$ with respect to $S$;
see \cite{moore_fast_2013}. Such a subset $A$ is called $\frac1n$-\emph{F\o lner set} corresponding to the ball $B_{G,S}\left( n\right).$
It is clear that the $\simeq $-equivalence class $\mathrm{F\o l}_{G}$ of this F\o lner function does not depend
on the chosen finite generating set $S$.

\begin{remark}\label{rem:folner}
There is a great flexibility in the choice of the definition of a F\o lner function as it is in the choice of the definition of
a F\o lner set. For instance, instead of the symmetric difference above one can take $\left\vert gA\setminus A\right\vert$ or 
instead of $g\in B_{G,S}\left( n\right)$ one can assume that $g\in S$, etc. We leave to the reader to check that all these natural variations lead to
$\simeq$-equivalent F\o lner functions and they do not depend on the choice of the finite generating set of $G$. 
\end{remark}

It follows from the proof that amenable groups are sofic---see for instance 
\cite[Example 4.4]{pestov_hyperlinear_2008}---that, for an amenable group $G$, we have:
$$
\mathcal{D}%
_{G,S}^{sof}\left( n\right) \leqslant \mathrm{F\o l}_{G,S}\left( 2n\right). $$ 
This gives $$
\mathcal{D}%
_{G,S}^{lin}\left( n\right) \leqslant \mathrm{F\o l}_{G,S}\left(n\right) \hbox{ and } \mathcal{D}%
_{G,S}^{hyp}\left( n\right) \leqslant \mathrm{F\o l}_{G,S}\left( 4n^2\right).$$

Basic examples of amenable groups include groups of subexponential growth and elementary amenable groups.
Among the latter are virtually nilpotent groups. 
It is not hard to estimate the F\o lner functions of such groups\footnote{We give direct estimates. A more careful study can be done for specific groups and classes of groups.} and, hence, to obtain the upper bounds for 
their metric profiles using the preceding inequalities. On the lower bound see Question~\ref{q:nilp}.

\begin{example}[Groups of subexponential growth]
\label{Example:subexponential}Suppose that $G$ has subexponential growth and 
$S$ is a finite generating set of $G$. Define $a_{n}$ to be the size of  the ball $%
B_{G,S}\left( n\right) $ for $n\in \mathbb{N}$. Since $G$ has subexponential
growth we have that $\lim_{n\to\infty}\sqrt[n]{a_{n}}=1$ and hence $\lim_{n\to\infty}\frac{%
a_{n+1}}{a_{n}}=1$ \cite[Lemma 6.11.1]{ceccherini-silberstein_cellular_2010}%
. Given $n\in \mathbb{N}$, let $m\in \mathbb{N}$ be so that $a_{k+1}/a_{k}\leqslant
\left( 1+\frac{1}{2na_{n}}\right) ^{\frac{1}{n}}$ for every $k\geqslant m$, then
we have that, for every $g\in B_{G,S}\left( n\right) $,%
\begin{equation*}
\frac{\left\vert gB_{G,S}\left( m\right) \backslash B_{G,S}\left( m\right)
\right\vert }{\left\vert B_{G,S}\left( m\right) \right\vert }\leqslant \frac{%
a_{n+m}-a_{n}}{a_{m}}\leqslant \frac{1}{2na_{n}}
\end{equation*}%
and hence 
\begin{equation*}
\frac{1}{\left\vert B_{G,S}\left( m\right)\right\vert }\sum_{g\in B_{G,S}\left( n\right) }\left\vert gB_{G,S}\left( m\right)
\bigtriangleup B_{G,S}\left( m\right) \right\vert \leqslant \frac{1}{n}\text{.}
\end{equation*}%
This shows that, for a group $G$ of subexponential growth, we have:
 $$\mathcal{D}_{G}^{sof}\left( n\right) \preccurlyeq\mathrm{F\o l}_{G}\left( n\right) \leqslant \min \left\{
m : a_{k+1}/a_{k}\leqslant \left(1+\frac{1}{2na_{n}}\right)^{\frac{1}{n}}\text{ for }k\geqslant
m\right\} .$$

\end{example}

\begin{example}[Virtually nilpotent groups]\label{ex:vnF}
\label{Example:virtually}Suppose that $G$ is a virtually nilpotent group
with finite generating set $S$. By Gromov's polynomial growth theorem \cite%
{gromov_groups_1981}, $G$ has polynomial growth. Let $d$ be the order of
polynomial growth of $G$. Then, using the notation of Example \ref%
{Example:subexponential}, for every $\varepsilon >0$ one has, for all
but finitely many $n\in \mathbb{N}$, $n^{d}/2\leqslant a_{n}\leqslant 2n^{d}$.
Therefore, we have that $a_{m+1}/a_{m}\leqslant 4\left( 1+1/m\right) ^{d}$ and $%
\left( 1+(2na_{n})^{-1}\right) ^{\frac{1}{n}}\geqslant \left( 1+n^{-\left(
d+1\right) }/4\right) ^{\frac{1}{n}}$. Hence, we get from Example \ref%
{Example:subexponential} that 
\begin{equation*}
\mathrm{F\o l}_{G,S}\left( n\right) \leqslant \left\lceil 4\left( \left( 1+\frac{1%
}{4n^{d+1}}\right) ^{\frac{1}{dn}}-1\right) ^{-1}\right\rceil \text{.}
\end{equation*}%
This gives that%
\begin{equation*}
\text{\textrm{F\o l}}_{G,S}\left( n\right) \leqslant 2^{dn+4}n^{d+1}\text{.}
\end{equation*}%
Indeed, observe that, for any number $x$,%
\begin{equation*}
\left( 1+2^{-n}x\right) ^{n}\leqslant 1+x
\end{equation*}%
and%
\begin{equation*}
\left( 1+x\right) ^{\frac{1}{n}}\geqslant 1+2^{-n}x\text{.}
\end{equation*}%
Therefore,%
\begin{equation*}
\left( 1+\frac{1}{4n^{d+1}}\right) ^{\frac{1}{dn}}-1\geqslant \frac{1}{%
2^{dn+2}n^{d+1}}
\end{equation*}%
and%
\begin{equation*}
4\left( \left( 1+\frac{1}{4n^{d+1}}\right) ^{\frac{1}{dn}}-1\right)
^{-1}\leqslant 2^{dn+4}n^{d+1}\text{.}
\end{equation*}

It follows that, for a virtually nilpotent group $G$, we have:
 $$\mathcal{D}_{G}^{sof}\left( n\right) \preccurlyeq \mathrm{F%
\o l}_{G}\left( n\right) \leqslant 2^{dn+4}n^{d+1}.$$
\end{example}

The two examples above use the fact that in these groups  subsequences of balls form a collection of F\o lner sets.
Since the growth of balls is prescribed one can erroneously expect to have this prescribed (e.g. subexponential or polynomial) behavior
of the corresponding F\o lner functions. However, our quantitative statements on F\o lner functions require more information than just
the knowledge of the asymptotic type of the growth function or of the classical isoperimetric profile of an amenable group. By Definition~\ref{def:app}, 
given $n\in\mathbb N$, one has to determine an exact dependence between 
the radius $N$ of the $\frac1n$-F\o lner set $B_{G,S}(N)$ corresponding to the ball $B_{G,S}(n)$ and $n$.
This dependence is exponential in our direct estimates above.

\subsection{LE-$\mathcal F$ growth and LEA profile}\label{sec:lea}
Suppose that $G$ is a residually finite group with a finite generating set $%
S $. Quantifying functions associated with residually finite groups have been
investigated, for example, in \cite%
{bou-rabee_quantifying_2010,bou-rabee_asymptotic_2011,bou-rabee_arbitrarily_2016,bou-rabee_residual_2014,bou-rabee_nilpotent}%
.  The most popular such a function is the \emph{residual finiteness growth} (also called,  the \emph{depth function}),
denoted by $F_{G,S}\left( n\right)$: it is defined to be the least integer $k>0$ such
that for any non-identity element $g\in B_{G,S}\left( n\right) $ there exists
a normal subgroup of $G$ of index at most $k$ that does not contain~$g$.

Our approach is closer in the spirit to another, less known,  function quantifying residual finiteness, the \emph{full residual finiteness growth}
(or the \emph{full depth function}), denoted by $\Phi_{G,S}\left( n\right)$: it is defined to be the least $k$ such that
there exists a normal subgroup $N\unlhd G$ of index at most $k$ that meets $%
B_{G,S}\left( n\right) $ only at the identity. It is clear that the $\simeq $%
-equivalence class of $\Phi_{G,S}(n)$ does not depend from the generating set $S$, and can be
denoted by $\Phi_{G}$. 

By Cayley's theorem, a finite quotient $G/N$ embeds into the symmetric group acting on the quotient itself:
$G/N\hookrightarrow \Sym(\vert G/N\vert)$.
Therefore, for a residually finite group $G$, we have:
$$
\mathcal{D}_{G,S}^{sof}\left( n\right) \leqslant \Phi_{G,S}(n),\quad \mathcal{D}_{G,S}^{hyp}\left(n\right) \leqslant \Phi_{G,S}(n),\quad \mathcal{D}_{G,S}^{lin}\left( n\right) \leqslant \Phi_{G,S}(n)\hbox{ and }\ 
$$
$$
\beta_{G,S}(n)\leqslant \mathcal{D}_{G, S}^{fin} (n)\leqslant \Phi_{G,S}(n).
$$

\begin{example}[Linear / Nilpotent / Virtually abelian groups]\label{ex:linear}
Finitely generated linear groups have at most exponential function $\Phi_{G,S}\left( n\right)$, see~\cite{systolicLin} (where the full residual finiteness growth  is termed the \emph{residual girth} and the \emph{normal systolic growth}). It follows 
 that all the main metric profiles, $\mathcal{D}_{G,S}^{hyp}\left( n\right), \mathcal{D}_{G,S}^{sof}\left( n\right), \mathcal{D}_{G,S}^{lin}\left( n\right)$ and $\mathcal{D}_{G, S}^{fin} (n)$,  are at most exponential. If, in addition, such a  group is not virtually nilpotent then it has at least exponential growth. We conclude that 
$\mathcal{D}_{G, S}^{fin} (n)$ is exponential whenever $G$ is a finitely generated linear group that is not virtually nilpotent. See also Question~\ref{q:lin}.

If $G$ is a finitely generated nilpotent group, then $\beta_{G,S}(n)\simeq \Phi_{G,S}(n)$ if and only if $G$ is virtually abelian~\cite{bou-rabee_nilpotent}. Hence, $\mathcal{D}_{G, S}^{fin} (n)\simeq \beta_{G,S}(n)\simeq \Phi_{G,S}(n)$ if and only if $G$ is virtually abelian. Classical examples of nilpotent groups which are not virtually abelian include discrete Heisenberg groups $H_{2n+1}=H_{2l+1}(\mathbb Z)$ of dimension $2l+1$ with $l\geqslant 1$.  By  the Bass-Guivarc'h formula for growth in terms of the derived series of a finitely generated nilpotent group, $\beta_{H_{2l+1},S}(n)\simeq n^{2l+2}$. The upper and lower central series of $H_{2l+1}$ coincide, then by ~\cite[Theorem 1]{bou-rabee_nilpotent}  we have $\Phi_{H_{2l+1},S}(n)\simeq n^{2(2l+1)}$. That is,  $ n^{2l+2} \preccurlyeq \mathcal{D}_{H_{2l+1}, S}^{fin} (n)\preccurlyeq  n^{2(2l+1)}$.

If $H_{2l+1}$ is $\mathcal F^{fin}$-stable, then by Corollary~\ref{cor:stab}\ref{cor:stab2}, we have $\mathcal{D}_{H_{2l+1}, S}^{fin} (n)\simeq n^{2(2l+1)}$. See Conjecture~\ref{conj:heis} and more generally Conjecture~\ref{conj:nilp}. See also Example~\ref{ex:stablegps}.
\end{example}

The above way of quantifying and the above estimates extend immediately from the class of residually finite groups 
to a more general setting of LE-$\mathcal F$ groups, discussed in Section~\ref{sec:lef}:
one can introduce the corresponding quantifying function $\Phi_{G,S}^\mathcal F\left( n\right)$  as follows.

\begin{definition}[LE-$\mathcal F$ growth]\label{def:lefgrowth}
Let $G$ by a finitely generated group with a finite
generating set $S$. 
The \emph{LE-$\mathcal F$ growth} of $G$ is the function $%
\Phi_{G,S}^{\mathcal{F}}\colon\mathbb{N}\rightarrow \mathbb{N}\cup \left\{
+\infty \right\} $ defined by
$$
\Phi_{G,S}^{\mathcal{F}}\left( n\right)=\min\left\{ {k\in \mathbb{N}} \mid \exists \hbox{ a group }  
G_{\alpha }\in\mathcal F  \hbox{ of size } k_{\alpha }=k \hbox{ with } B_{G,S}(n)\hookrightarrow G_\alpha\right\},
$$
where $B_{G,S}(n)\hookrightarrow G_\alpha$ is a \emph{monomorphism on the ball},
that is, it preserves the algebraic operation of $G$ on the elements from $B_{G,S}(n)$ (i.e. it is a homomorphism on those elements) and it is an injection.
\end{definition}

Observe that this LE-$\mathcal F$ growth is nothing else that our $\mathcal F_{\{0,1\}}$-profile
and, for an LE-$\mathcal F$ group $G$, we have:
$$
\mathcal{D}_{G,S}^{\mathcal F}\left( n\right) \leqslant \Phi_{G,S}^{\mathcal{F}}\left( n\right)=\mathcal{D}_{G,S}^{\mathcal F_{\{0,1\}}}\left( n\right).
$$
This applies to an arbitrarily fixed family $\mathcal F=\{(G_{\alpha},d_{\alpha}, k_{\alpha}, \varepsilon _{\alpha })\}_{\alpha\in I }$.

When $\mathcal F=\mathcal{F}^{fin}_{\left\{
0,1\right\}}$, that is, for an LEF group $G$, the inequalities above  specify to 
$$
\beta_{G,S}(n)\leqslant \mathcal{D}_{G,S}^{fin}\left( n\right) \leqslant \Phi_{G,S}^{\mathcal{F}^{fin}_{\left\{
0,1\right\}}}\left( n\right)=\mathcal{D}_{G,S}^{\mathcal F^{fin}_{\{0,1\}}}\left( n\right).
$$
It would be interesting to find an example of an LEF group with the second inequality being strict, cf. Question~\ref{q:lef}.

When $\mathcal F=\mathcal F^a_{\{0,1\}}$ is a family of amenable groups, every amenable group $G_\alpha\in \mathcal F^a_{\{0,1\}}$
generated by a finite generating set $S_\alpha$
has the associated F\o lner function $\mathrm{F\o l}_{G_\alpha, S_{\alpha}}(n).$ Therefore,
for an $\mathcal F^a_{\{0,1\}}$-approximable group we extend our quantifying viewpoint further. 

\begin{definition}(LEA-$\mathcal F^{a}$ profile)
Let $G$ by a finitely generated group with a finite
generating set $S$. 
The \emph{LEA-$\mathcal F^{a}$ profile} of $G$ is the function $%
LEA_{G,S}^{\mathcal F^{a}}\colon\mathbb{N}\rightarrow \mathbb{N}\cup \left\{
+\infty \right\} $ defined by
$$
LEA_{G,S}^{\mathcal F^{a}}\left( n\right)=\min\left\{ \mathrm{F\o l}_{G_\alpha, S_\alpha}(n) \mid \exists \hbox{ a group }  
G_{\alpha }\in\mathcal F^{a}_{\{0,1\}} \hbox{  
   with } B_{G,S}(n)\hookrightarrow G_\alpha\right\}.
$$
\end{definition} 

This allows to study F\o lner like profile functions for non-amenable groups which possess  `exact' approximations by amenable groups.
If $\mathcal F^{a}_{\{0,1\}}$ consists of \emph{all} amenable groups, we denote the corresponding LEA profile by $LEA_{G,S}\left( n\right)$.
Also, if $\mathcal F^{a}_{\{0,1\}}$ consists of amenable quotients of $G$, 
this gives the \emph{residually amenable profile}, denoted by $RA_{G,S}(n)$.
Without loss of generality, we assume no any relationship between the generating sets $S$ and $S_\alpha$. Also,
the assumption on the finiteness of $S_\alpha, \alpha\in I$, is not necessary. If the approximating groups $G_\alpha$ are not finitely generated,
given a group $G_\alpha$ with  $B_{G,S}(n)\hookrightarrow G_\alpha$, we can consider
a subgroup of $G_\alpha$ generated by finitely many images under this injection of elements from $B_{G,S}(n)$ and take the F\o lner function of this finitely generated subgroup of $G_\alpha$.  Properties of usual F\o lner functions immediately extend to our LEA profile and RA profile. In particular,
 the $\simeq $-equivalence class of $LEA_{G,S}$, respectively of $RA_{G,S}$, does not depend on the choice of the finite
generating set $S$.

For a  group $G$, and a family $\mathcal F\supseteq \mathcal F^{a}\supseteq\{\hbox{ amenable quotients of } G\}$, we have:
$$
\mathcal{D}_{G,S}^{\mathcal F}\left( n\right)\leqslant LEA_{G,S}\left( n\right)\leqslant RA_{G,S}(n)\leqslant \mathrm{F\o l}_{G, S}(n),
$$
where $\mathrm{F\o l}_{G, S}(n)=\infty$ whenever $G$ is non-amenable and the dimensions of $G_\alpha\in \mathcal F^{a}$ approximating $B_{G,S}(n)$ satisfy $k_\alpha\leqslant 
\mathrm{F\o l}_{G_\alpha, S_\alpha}(n).$ See also Question~\ref{q:ra}.

Subsequent to the present work and partially on the suggestion of the first author, the functions LE-$\mathcal F^{fin}$, $LEA_{G,S}\left( n\right)$ and $RA_{G,S}(n)$
have been also studied in~\cite{federico} and~\cite{matteo}.

\subsection{Metric profiles of stable metric approximations}\label{subsec:stab}
As above, $G$ denotes a group generated by a finite set $S=X\sqcup X^{-1}$, with, say, $X=\{x_1, \ldots, x_m\}$, subject to a finite set of relators $R\subseteq F_m=F_X$ and
$\mathcal{F}=\left( G_{\alpha }, d_{\alpha }, k_{\alpha }, \varepsilon _{\alpha }\right)
_{\alpha \in I}$ is an approximating family of $G$. 

If $r\in F_m$ and $g_1,\ldots, g_m$ are elements in a group $H$, we denote by $r(g_1,\ldots, g_m)\in H$ the image of $r$ under the unique group homomorphism 
$F_m\to H$ such that $x_i\mapsto g_i$.

\begin{definition}[Solution and almost solution]
Elements $g_1^\alpha, \ldots, g_m^\alpha\in G_\alpha$ are a \emph{solution} of $R$ in $G_\alpha$  if
\begin{equation*}\label{eq:sys}
r(g_1^\alpha, \ldots, g_m^\alpha)=e_\alpha, \, \forall r\in R,
\end{equation*}
 where $e_\alpha$ denotes the identity of $G_\alpha$.

Elements $g_1^\alpha, \ldots, g_m^\alpha\in G_\alpha$ are a \emph{$\delta$-solution} of $R$ in $G_\alpha$, for some $\delta>0$, if
\begin{equation*}\label{eq:almost}
d_\alpha\left(r(g_1^\alpha, \ldots, g_m^\alpha),e_\alpha\right)<\delta, \, \forall r\in R.
\end{equation*}
\end{definition}

The following notion is due to Arzhantseva-P\u{a}unescu~\cite{arzhantseva_almost_2015}, see also~\cite{arzhantseva_centr_2016} for a more general setting.

\begin{definition}[$\mathcal F$-stable groups]\label{def:stab}
The set $R$ is \emph{$\mathcal F$-stable} if $\forall \varepsilon>0\, \exists \delta>0 \, \forall \alpha\in I\, \forall g_1, \ldots, g_m\in G_\alpha$ a $\delta$-solution
of $R$, there exist $\tilde g_1, \ldots, \tilde g_m \in G_\alpha$ a solution of $R$ such that $d_\alpha(g_i, \tilde g_i)<\varepsilon.$

The group $G = F_m/ \langle R\rangle$ is called \emph{$\mathcal F$-stable} if its set of relator words $R$ is $\mathcal F$-stable.
\end{definition}

The definition of $\mathcal F$-stability does not depend on the particular choice of finite presentation of the
group: Tietze transformations preserve stability as the metrics $d_\alpha$ are bi-invariant, see~\cite[Section 3]{arzhantseva_almost_2015}.
The following theorem is due to Arzhantseva-P\u{a}unescu, cf.~\cite[Definition 4.1, Theorem 4.2 and lines after it, Theorem 4.3]{arzhantseva_almost_2015}, see again~\cite{arzhantseva_centr_2016} for a more general variant.

\begin{theorem}\label{thm:stab}
Assume that $(\varepsilon_\alpha)_{\alpha\in I}$ is constantly equal to a real number $\varepsilon>0$. Then the following holds.
\begin{itemize}
\item[(1)] The set $R$ is $\mathcal F$-stable if and only if any group homomorphism $\iota\colon G\rightarrow \prod_{%
\mathcal{U}}\left( G_{\alpha },d_{\alpha }\right) $ lifts to $\prod_{\alpha \in I}G_{\alpha}$.
\item[(2)]  If $G$ is $\mathcal F$-approximable and $\mathcal F$-stable, then $G$ is LE-$\mathcal F$, equivalently, fully residually $\mathcal F$.
\end{itemize}
\end{theorem}

A homomorphism $\iota$ is not necessarily injective and the lifting property means that there exists 
$g^i_\alpha\in G_{\alpha}, i = 1,\ldots,m$ such that $\{g^1_\alpha,\ldots,g^m_\alpha\}$ is a solution 
of $R$ for  any $\alpha\in I$ and $\iota ( x_i ) = \prod_{%
\mathcal{U}}g^i_\alpha,$ see~\cite[Definition 4.1]{arzhantseva_almost_2015}.
The equivalence between the LE-$\mathcal F$ property and the fully residually $\mathcal F$ property is because, in this section, $G$
is assumed to be finitely presented.

\begin{corollary}\label{cor:stab}
Under the hypothesis above we have the following inequalities.
\begin{enumerate}[label=(\arabic*)]
\item  If the set $R$ is $\mathcal F$-stable, then  
$$
\mathcal{D}_{G,S}^{\mathcal F_{\{0,1\}}}\left( n\right) \leqslant \mathcal{D}_{G,S}^{\mathcal F}\left( n\right).
$$
\item If $G$ is $\mathcal F$-approximable and $\mathcal F$-stable, then 
$$
\Phi_{G,S}^{\mathcal{F}}\left( n\right)\simeq\mathcal{D}_{G,S}^{\mathcal F_{\{0,1\}}}\left( n\right)\simeq\mathcal{D}_{G,S}^{\mathcal F}\left( n\right).
$$
\label{cor:stab2}
\end{enumerate}
\end{corollary}
In particular, for an  $\mathcal F^{sof}$-stable sofic group $G$ one has
$$
\Phi_{G,S}\left( n\right)\simeq\mathcal{D}_{G,S}^{sof}\left( n\right).$$ 
Similarly, for an $\mathcal F^{hyp}$-stable hyperlinear, respectively $\mathcal F^{lin}$-stable linear sofic, and yet, respectively $\mathcal F^{fin}$-stable weakly sofic group $G$ one has
$
\Phi_{G,S}\left( n\right)\simeq\mathcal{D}_{G,S}^{hyp}\left( n\right),$ respectively $\Phi_{G,S}\left( n\right)\simeq\mathcal{D}_{G,S}^{lin}\left( n\right),$ and yet, respectively $\Phi_{G,S}\left( n\right)\simeq\mathcal{D}_{G,S}^{fin}\left( n\right).$ Here we imply that, under the hypothesis of Corollary~\ref{cor:stab}\ref{cor:stab2}, the corresponding LE-$\mathcal F$ growths coincide, e.g. 
$\Phi_{G,S}^{\mathcal{F}^{sof}}(n)\simeq \Phi_{G,S}(n)$, etc.

The assumption on finite presentation of $G$ can be relaxed, cf.~\cite[Remark 2.10]{arzhantseva_centr_2016}. Moreover, the above statements
can be generalized to a wider context of \emph{constraint} metric approximations~\cite{arzhantseva_centr_2016}. 

Corollary~\ref{cor:stab}  allows us to explicit various metric profiles for groups known to be stable. 
\begin{example}[$\mathcal F$-profiles of $\mathcal F$-stable groups]\label{ex:stablegps}\strut 

(1) Let $G$ be a finitely generated virtually abelian group and ${\rm rank}\,G$ denote the rank of any finite index free-abelian subgroup of $G$. Then $$\mathcal{D}_{G,S}^{sof}\left( n\right)\simeq\mathcal{D}_{G,S}^{hyp}\left( n\right)\simeq\mathcal{D}_{G,S}^{fin}\left( n\right) \simeq n^{{\rm rank}\, G}.$$  Indeed, we use~\cite{arzhantseva_almost_2015} for $\mathcal F^{sof}$-stability,~\cite{glebsky2010commuting} for $\mathcal F^{hyp}$-stability,  Example~\ref{ex:linear} for  $\mathcal{D}_{G,S}^{fin}\left( n\right)$,~\cite{bou-rabee_nilpotent} for $\Phi_{G,S}\left( n\right)\simeq n^{{\rm rank}\,G}$ (or simply argue as in Example~\ref{ex:ab}) and Corollary~\ref{cor:stab}\ref{cor:stab2} together with Proposition~\ref{Proposition:finite-index}\ref{i:fi} for conclusion. See Conjecture~\ref{conj:ablin} on $\mathcal{D}_{G,S}^{lin}\left( n\right)$.

(2) Let $G$ be a finitely generated virtually nilpotent group so that a finite-index nilpotent subgroup $H\leqslant G$ satisfies $[Z(H):\gamma_c(H)]<\infty$, where $Z(H)$ and $\gamma_c(H)$ denote  the center and the last nontrivial term of the lower central series of $H$ (i.e. $c$ is the nilpotency class of $H$).  Let $\dim H$ denote the number of infinite cyclic factors in a composition series of $H$ with cyclic factors.
Then $$\mathcal{D}_{G,S}^{sof}\left( n\right)\simeq n^{\dim H}.$$ Indeed, we apply Proposition~\ref{Proposition:finite-index}\ref{i:fi} to restrict to a nilpotent group $H$ as above, then~\cite{oren_irs} for $\mathcal F^{sof}$-stability,~\cite{bou-rabee_nilpotent} for $\Phi_{H,S}\left( n\right)\simeq n^{\dim H}$ and Corollary~\ref{cor:stab}\ref{cor:stab2} for conclusion. In particular, for the $(2l+1)$-dimensional Heisenberg group we have $\mathcal{D}_{H_{2l+1},S}^{sof}\left( n\right) \simeq n^{2(2l+1)}$.

The assumption $[Z(H):\gamma_c(H)]<\infty$ is required by \cite[Theorem 1]{bou-rabee_nilpotent}.  For an arbitrary virtually nilpotent group, analogous conclusions hold  using~\cite[Theorem 2]{bou-rabee_nilpotent} which provides a polynomial upper bound on $\Phi_{G,S}\left( n\right)$, and hence, by Corollary~\ref{cor:stab}\ref{cor:stab2}, on the sofic profile $\mathcal{D}_{G,S}^{sof}\left( n\right)$. This improves the upper bound from Example~\ref{ex:vnF}.

Furthermore, all virtually \{polycyclic-by-finite\} groups which are not virtually nilpotent have exponential $\mathcal{D}_{G,S}^{sof}\left( n\right)$. For we use Proposition~\ref{Proposition:finite-index}\ref{i:fi} to restrict to a polycyclic-by-finite group $H$, then~\cite{oren_irs} for $\mathcal F^{sof}$-stability,~\cite{systolicLin}  for $\Phi_{H,S}\left( n\right)\simeq 2^n$ (note that every polycyclic-by-finite group is linear~\cite[Section 5.C]{segal_polyc}) and Corollary~\ref{cor:stab}\ref{cor:stab2} for conclusion.

On the hyperlinear profile, we have $\mathcal{D}_{H_{2l+1},S}^{hyp}\left( n\right) \simeq n^{2(2l+1)}$ for the $(2l+1)$-dimensional Heisenberg group, since it is $\mathcal F^{hyp}$-stable~\cite{hadwin_HS} (the argument extends from 3-dimensional to $(2l+1)$ dimensional Heisenberg group) and $\Phi_{H_{2l+1},S}\left( n\right)\simeq n^{2(2l+1)}$~\cite[Theorem 1]{bou-rabee_nilpotent}. 

See Conjecture~\ref{conj:heis} on other metric profiles of the Heisenberg groups and Conjecture~\ref{conj:nilp} on arbitrary finitely generated virtually nilpotent groups.
\end{example}

\subsection{Other metric profiles}  Subsequent to our work other metric profiles have been introduced, with alternative (not equivalent!) to Definition~\ref{def:profile} formulations, and restricted to certain metric approximations. Notably, a sofic profile was introduced in~\cite{cornulier:cremona} and, by analogy, a hyperlinear profile in~\cite{hyperProf}. In both cases, 
the formulation is `transversal' to our line of thought. Indeed, one can parametrize an $\mathcal F$-approximation by two parameters $m$ and $n$ (instead of one parameter $n$ in Definition~\ref{def:app}): $m$ being the radius of the ball to be approximated and $1/n$ being the constant involved in the definitions of almost homomorphism and uniform injectivity on that ball. This can seem to give a greater flexibility but in fact provides an equivalent definition of an $\mathcal F$-approximation. Even so, this yields distinct approaches to quantifying when one prescribes different constraints on the corresponding quantifying function $\mathcal{D}_{G,S}^{\mathcal{F}}\left( m, n\right) $ of two variables. Our Definition~\ref{def:app}, and hence Definition~\ref{def:profile}, take $m=n$ so that we consider the values of the quantifying function on the diagonal of $m$-$n$ plane (viewing, for example, $m$ on the horizontal and $n$ on the vertical axes). In contrast, both~\cite{cornulier:cremona} and~\cite{hyperProf} choose to fix $m$ and consider such a function on the vertical $n$-line, when $n$ varies. A careful reader is invited to pay attention to such differences and variations in the existing terminology. See also the notion of sofic dimension~\cite{dykema_sd} which mirrors an `orthogonal' concept of subgroup growth we alluded to in the introduction.
 
\section{Metric profile and group-theoretical constructions}

In this section, we observe how  the $\mathcal{F}$-profile behave with respect to
various group-theoretical constructions such as taking subgroups, direct and free products, and restricted wreath products. 

\subsection{Subgroups}
It follows from Definition~\ref{def:app} that every subgroup of an $\mathcal F$-approximable group is $\mathcal F$-approximable and from
definition of an LE-$\mathcal F$ group in Section~\ref{sec:lef} that a subgroup of an LE-$\mathcal F$ group is an LE-$\mathcal F$ group.
Quantifying these statements we obtain the following easy but instructive result.

\begin{proposition}
\label{Proposition:finite-index} Let $G\ $ be a finitely generated $\mathcal F$-approximable group and $H$ be a finitely generated subgroup of $G$. Then the following holds.
\begin{enumerate}[label=(\roman*)]
\item $\mathcal{D}%
_{H}^{\mathcal F}\leqslant \mathcal{D}_{G}^{\mathcal F}$. \label{i:sg}
\item If $H$ has finite
index in $G$, then $\mathcal{D}_{G}^{sof}\simeq \mathcal{D}_{H}^{sof}, \mathcal{D}_{G}^{hyp}\simeq \mathcal{D}_{H}^{hyp},$ and
$\mathcal{D}_{G}^{lin}\simeq \mathcal{D}_{H}^{lin}$.\label{i:fi}
\item If $G$ is an LE-$\mathcal F$ group, then $\Phi_{H}^{\mathcal{F}}\left( n\right)\leqslant \Phi_{G}^{\mathcal{F}}\left( n\right).$\label{prop:fii3}
\item
If $G$ is an LEA group, then $LEA_H(n)\leqslant LEA_G(n).$
\end{enumerate}

\end{proposition}

\begin{proof}
The first assertion is clear and it holds for an arbitrary $%
\mathcal{F}$-profile. The statements about LE-$\mathcal F$ and LEA profiles are by definitions.

We now prove the second assertion. By~\ref{i:sg} above, it suffices to show that $\mathcal{D}_{G}^{sof}\preccurlyeq \mathcal{D}_{H}^{sof}$. Let $\ell $ be the
index of $H$ in $G$, and $T=\left\{ g_{1},\ldots ,g_{\ell }\right\} $ be a
choice of representatives for left cosets of $H$ in $G$. Observe that for
every $g\in G$ and $i\leqslant \ell $ there exist unique $\alpha _{g}\left(
i\right) \leqslant \ell $ and $h_{g,i}\in H$ such that $gg_{i}=g_{\alpha
_{g}\left( i\right) }h_{g,i}$. From the uniqueness of such a representation
one can deduce that $\alpha _{g}\circ \alpha _{k}=\alpha _{gk}$ and $%
h_{gk,i}=h_{g,\alpha _{k}\left( i\right) }h_{k,i}$. Furthermore the map $%
i\mapsto \alpha _{g}\left( i\right) $ is a permutation of $\left\{ 1,\ldots
,\ell \right\} $. Suppose now that $\varphi\colon H\rightarrow \Sym(m)$ is an $%
\left( n,1\right) $-approximation of $H$. We define a map  $\psi\colon G\rightarrow \Sym(\ell m)$ by identifying $%
\Sym(\ell m)$ with the group of permutations of $\left\{ 1,\ldots ,\ell
\right\} \times \left\{ 1,\ldots , m\right\} $ and defining $\psi \left(
g\right) $ to be the map $\left( i,j\right) \mapsto \left(
\alpha_{g}\left( i\right) ,\varphi \left( h_{g,i}\right) \left( j\right) \right).$
It is straightforward to check that  for every $g\in G$  the map $ \psi \left(
g\right) $ is a permutation and that $\psi$ is an $\left( n,1\right) $%
-approximation of $G$ of dimension $\ell m$. For the latter, we use a general fact that given two permutations $\sigma\in \Sym(m)$ and $\tau\in\Sym(q)$, the normalized Hamming distance of the direct sum $\sigma\oplus \tau\in \Sym(m+q)$ satisfies
$$
d_{\rm{Ham}}(\sigma\oplus \tau, e_{\Sym(m+q)})=\frac{md_{\rm{Ham}}(\sigma, e_{\Sym(m)}) + qd_{\rm{Ham}}(\tau, e_{\Sym(q)})}{m+q}.
$$
This shows the required $\mathcal{D}_{G}^{sof}\preccurlyeq \mathcal{D}_{H}^{sof}$.
The proofs for the hyperlinear and linear sofic profiles are analogous, using corresponding properties of the normalized Hilbert-Schmidt and rank distances, cf.~\cite[Proposition 2.2]{arzhantseva_linear_2012}. Namely, we use a general fact that given two unitary matrices $u\in U(m)$ and $v\in U(q)$, the normalized Hilbert-Schmidt distance of the block-diagonal matrix $u\oplus v\in U(m+q)$ satisfies
$$
d_{\rm{HS}}^2(u\oplus v, e_{U(m+q)})=\frac{md^2_{\rm{HS}}(u, e_{U(m)}) + qd^2_{\rm{HS}}(v, e_{U(q)})}{m+q}.
$$
Similarly, given $u\in GL(m,\mathbb K)$ and $v\in GL(q,\mathbb K)$, the normalized rank distance of the block-diagonal matrix $u\oplus v\in GL(m+q, \mathbb K)$ satisfies
$$
d_{\rm{rank}}(u\oplus v, e_{GL(m+q, \mathbb K)})=\frac{md_{\rm{rank}}(u, e_{GL(m,\mathbb K)}) + qd_{\rm{rank}}(v, e_{GL(q,\mathbb K)})}{m+q}.
$$

\end{proof}

\begin{remark}
Proposition~\ref{Proposition:finite-index}\ref{i:fi} extends to other metric $\mathcal F$-profiles. The proof proceeds as above constructing $\mathcal F$-approximations on $G$ induced (as induced representations) by a given $\mathcal F$-approximation of $H$. For this to work, assumptions on the family $\mathcal F$ and the distances $d_\alpha$ are required. For example, one assumes that $G_\alpha\oplus G_\beta$ is defined and is isomorphic to some $G_\gamma$ and the distances $d_\alpha, d_\beta, d_\gamma$ satisfy
that $f(d_\gamma)$ is a convex combination of $f(d_\alpha)$ and $f(d_\beta)$ for some function $f$. We call such functions $f$ \emph{diagonally block-convex}, which include so-called \emph{diagonally block-additive} functions where the convexity is given by the sum: $f(d_\gamma)=f(d_\alpha)+f(d_\beta)$.
In our  proof above, $f(x)=x$ for the normalized Hamming and rank distances and $f(x)=x^2$ for the Hilbert-Schmidt distance. Obviously, we can take $f(x)=x^p, 1\leqslant p\leqslant \infty$ which yields the convex equality as above for the $p$-Schatten norm. Therefore, we have the result as in Proposition~\ref{Proposition:finite-index}\ref{i:fi} for the metric approximations by matrices endowed with the normalized $p$-Schatten norm, for $1\leqslant p\leqslant \infty$, and hence, for the normalized trace and operator norms (the rank of a matrix can be viewed as its $p$-Schatten norm for $p=0$). Similarly, $\mathcal{D}_{G}^{fin}\simeq \mathcal{D}_{H}^{fin}$, whenever the distances $d_\alpha$ on the finite approximating groups admit some diagonally block-convex function as above.
\end{remark}

\subsection{Direct and free products}\label{sec:direct}
It follows from Proposition \ref{Proposition:finite-index} that for every $n\in \mathbb N$ we have $$\max\{\mathcal{D}^{\mathcal F}%
_{G}\left( n\right) , \mathcal{D}^{\mathcal F}%
_{H}\left( n\right) \} \leqslant \mathcal{D}^{\mathcal F}%
_{G\times H}\left( n\right), \hbox{ \strut } \max\{\Phi_{G}^{\mathcal{F}}\left( n\right),\Phi_{H}^{\mathcal{F}}\left( n\right)\} \leqslant \Phi_{G\times H}^{\mathcal{F}}\left( n\right) \hbox{ and }$$ 
$$\max\{LEA_{G}\left( n\right), LEA_{H}\left( n\right)\} \leqslant LEA_{G\times H}\left( n\right).$$ 
On the other hand,  taking a family $\mathcal F\in \left\{ \mathcal F^{sof},  \mathcal F^{hyp},  \mathcal F^{fin},   \mathcal F^{fin}_{cc}, \mathcal F^{lin}, 
\mathcal F^{fin}_{\{0,1\}}, \mathcal F^{a}_{\{0,1\}}\right\}$,
the known proofs that if $G$ and $H$ are $\mathcal F$-approximable, then
$G\times H$ is again $\mathcal F$-approximable (see, for example, \cite[Theorem 1]{elek_sofic_2006}  for the proof for sofic groups) give
$$
\mathcal{D}^{\mathcal F}%
_{G\times H}\left( n\right) \preccurlyeq \mathcal{D}^{\mathcal F}_{G}\left( n\right) 
\mathcal{D}^{\mathcal F}_{H}\left( n\right). 
$$
Analogously, for LE-$\mathcal F$ groups $G$ and $H$, we have, under the assumption that $\mathcal F$ is closed under taking direct product of two groups,
$$
\Phi_{G\times H}^{\mathcal{F}}\left( n\right)  \preccurlyeq \Phi_{G}^{\mathcal{F}}\left( n\right)\Phi_{H}^{\mathcal{F}}\left( n\right).
$$
It is also not hard to check that, for LEA groups $G$ and $H$, for the F\o lner type profile, we have:
$$
LEA_{G\times H}\left( n\right)\preccurlyeq LEA_{G}\left( n\right)LEA_{H}\left( n\right).
$$
The only point in the above estimates is to choose a suitable finite generating set of the direct product $G\times H$ starting from finite generating sets 
$S_G$ and $S_H$. A natural choice is satisfactory: one can take $S_{G\times H}= S_G\times\{ e_H\} \sqcup \{e_G\}\times S_H.$ 

The case of free products is less understood. Some of the classes of $\mathcal F$-approximable groups are not known to
be closed under taking free products, e.g. it is not yet proved for weak soficity and for weak hyperlinearity.
As usual, we focus on the quantifying aspects.

Let $G$ and $H$ be sofic groups, and $F_r$ denote a free group of rank $r\geqslant 2$. It follows from Proposition~\ref{Proposition:finite-index}\ref{prop:fii3} that
$\Phi_{F_{r}}\left( n\right)\leqslant \Phi_{F_{2}}\left( n\right)$. This and Elek-Szab\'o's proof of soficity of the free product $G\ast H$~\cite[%
Theorem 1]{elek_sofic_2006} give the estimate $$\mathcal{D}^{sof}_{G\ast H}\left(
n\right) \preccurlyeq \mathcal{D}^{sof}_{G}\left( n\right) \mathcal{D}^{sof}_{H}\left(
n\right) \Phi_{F_{2}}\left( n\right).$$

Analogously, a free product of linear sofic groups is linear sofic~\cite[Theorem 5.6]{stolz2013properties} and one might extract an upper bound on $\mathcal{D}^{lin}_{G\ast H}\left(
n\right)$ whenever $G$ and $H$ are linear sofic, see Conjecture~\ref{conj:lsfree}. 

A free product of hyperlinear groups is hyperlinear~\cite{popa_free,voiculescu_fp} or~\cite{Brown_fp} but a meaningful upper bound on $\mathcal{D}^{hyp}_{G\ast H}\left(
n\right)$, whenever $G$ and $H$ are hyperlinear, seems unknown, see Question~\ref{q:hypfree}.

For residually amenable $G$ and $H$, an upper bound on $RA_{G\ast H}(n)$ is obtained in~\cite[Theorem~6.2.7 and lines before Example~6.2.9]{federico}. See Question~\ref{q:leffp} and Question~\ref{q:leafp}.

\subsection{Extensions by amenable groups}

Let $G$ be a group, $A$ be a finite set of size $n$, and $G^A$ be the group of all functions from $A$ to $G$. Let $%
\mathrm{Sym}\left( A\right) $ be the group of permutations on $A$, which is
clearly isomorphic to the group $\Sym(n)$ of permutations on $\left\{
1,2,\ldots ,n\right\} $. 

The 
\emph{permutation wreath product } is the semidirect product $\mathrm{Sym}%
\left( A\right) \ltimes G^{A}$ with respect to the action of $\mathrm{Sym}%
\left( A\right) $ on $G^{A}$ that permutes the coordinates. An element of $%
\mathrm{Sym}\left( A\right) \ltimes G^{A}$ is represented by  a pair $\left(
\sigma ,b\right) $ where $\sigma \in \mathrm{Sym}\left( A\right) $ and $b\in
G^{A}$ is a function $b\colon A\to G$. When $G$ is a bi-invariant metric group, one can endow $\mathrm{Sym}%
\left( A\right) \ltimes G^{A}$ with a canonical bi-invariant metric, which
was defined in \cite[Section 5]{holt_closure} as follows. If $\sigma
_{0},\sigma _{1}\in \mathrm{Sym}\left( A\right) $ and $b_{0},b_{1}\in G^{A}$%
, then%
\begin{equation*}
\begin{split}
d_{\mathrm{Sym}\left( A\right) \ltimes G^{A}}\left( \left( \sigma
_{0},b_{0}\right) ,\left( \sigma _{1},b_{1}\right) \right) = \frac{1}{n}\sum
\left\{ d_{G}\left( b_{0}\left( a\right) ,b_{1}\left( a\right) \right) :a\in
A,\sigma _{0}\left( a\right) = \sigma _{1}\left( a\right) \right\} +\\ d_{%
\mathrm{Sym}\left( A\right) }\left( \sigma _{0},\sigma _{1}\right),
\end{split} 
\end{equation*}%
where $d_{\mathrm{Sym}\left( A\right) }$ denotes the normalized Hamming
distance on $\mathrm{Sym}\left( A\right) $.

We now introduce a useful notion of controlled F\o lner function, which is
implied by the controlled F\o lner sequence introduced in \cite[%
Section 3.2]{de_cornulier_isometric_2007}.

\begin{definition}[Controlled F\o lner function]
Let $G$ be a finitely-generated amenable group with finite symmetric generating set $S$.
The \emph{controlled F\o lner function} of $G$ with respect to $S$, denoted by  $\mathrm{F\o l}_{G,S}^{\emph{%
con}}\left( n\right) $, is defined to be the smallest $k\geqslant n$ such that there exists a
subset $A$ of $B_{G,S}\left( k\right) $ of size at most $k$ with the
property that $$\sum_{g\in B_{G,S}\left( n\right) }\left\vert
gA\bigtriangleup A\right\vert \leqslant \left\vert A\right\vert /n.$$ 
\end{definition}

As usual, choosing a different finite generating set would yield a function with the
same asymptotic type of growth. We let $\mathrm{F\o l}_{G}^{\emph{con}}$ be the corresponding $%
\simeq $-equivalence class of functions.

Let now $H$ be a subgroup of $G$ generated by a finite set $R$.
We now recall the notion of \emph{distortion } for the subgroup $H$ in $G$ as defined in \cite[Chapter
3]{gromov_asymptotic_1993}. 

\begin{definition}[Distortion function]
The\emph{\
distortion function} of $H$ in $G$ with respect to $R$ and $S$, denoted by $\Delta _{H\leqslant G,R,S}\left( n\right)$, is defined to be
the smallest $k\geqslant n$ such that $H\cap B_{G,S}\left( n\right) $ is
contained in $B_{H,R}\left( k\right) $. 
\end{definition}

Since different choices of
generating sets yield functions with the same asymptotic type of growth, the $\simeq $-equivalence class of 
$\Delta _{H\leqslant G}$ of the function defined above is well-defined.

Let $\mathcal F$ be an approximating family  $\mathcal{F}=\{(G_{\alpha},d_{\alpha}, k_{\alpha}, \varepsilon _{\alpha })\}_{\alpha\in I}$ such that for every 
$\alpha \in I$, $\varepsilon _{\alpha }$ is equal to a fixed strictly
positive constant, which up to normalization can be assumed to be equal to $1$.
We furthermore assume that for every $\alpha \in I$ and a finite set $A$ there
exists $\beta \in I$ such that $k_{\beta }\leqslant \left\vert A\right\vert
k_{\alpha }$, and the permutation wreath product $\mathrm{\mathrm{Sym}}%
\left( A\right) \ltimes G_{\alpha }^{A}$ is isometrically isomorphic (when
endowed with the canonical bi-invariant metric $d_{\mathrm{Sym}\left( A\right) \ltimes G_{\alpha}^{A}}$ described above) to $%
G_{\beta }$. For example, this applies when $I=\mathbb{N}$ and, for every $%
n\in \mathbb{N}$, $G_{n}$ is the permutation group $\Sym(n)$ endowed with the
normalized Hamming distance. The following result can be seen as a
quantified version of \cite[Theorem 5.1]{holt_closure}.

\begin{theorem}
Let $G$ be a finitely generated group and  $\mathcal F$ be a family as above, $N$ be a normal subgroup
of $G$ such that the quotient $G/N$ is amenable. Then we have: 
$$\mathcal{%
D}_{G}^{\mathcal{F}}\left( n\right) \preccurlyeq \mathrm{F\o l}%
_{G/N}^{\,\emph{con}}\left( n\right) \mathcal{D}_{N}^{\mathcal{F}}\left(\Delta
_{N\leqslant G}(\mathrm{F\o l}_{G/N}^{\,\emph{con}}\left( n\right) )\right).$$
\end{theorem}

\begin{proof}
For $g\in G$ we denote by $\overline{g}$ the image of $g$ under the quotient
map $G\twoheadrightarrow G/N$. Similarly, if $A$ is a subset of $G$, then we let $%
\overline{A}$ be the collection $\left\{ \overline{g}:g\in A\right\} $. Fix
a finite symmetric subset $T$ of $G$ such that $\overline{T}$ is a
generating set of $G/N$. Fix also a right inverse $\sigma\colon G/N\rightarrow G$
for the quotient map $G\twoheadrightarrow G/N$ such that $\sigma (B_{G/N,\overline{T%
}}(n))\subseteq B_{G,T}(n)$ for every $n\in \mathbb{N}$. In particular, this
implies that $\sigma \left( \overline{e}_{G}\right) =e_{G}$ and $\sigma
\left( \overline{g}\right) =g$ for $g\in T$. Let now $R$ be a finite
symmetric generating subset of $N$ with the property that, for every $t\in T$%
, $t^{-1}\sigma \left( \overline{t}\right) \in N$. Finally, let $S$ be the
finite symmetric generating set $R\cup T$ for $G$.

Let $n$ be a natural number. Observe that $\overline{B_{G,S}\left( n\right) }%
\subseteq B_{G/N,\overline{T}}(n)$ for every $n\in \mathbb{N}$. Set $k:=
\mathrm{F\o l}_{G/N,\overline{T}}^{\emph{con}}\left( 10n\right) \geqslant 10n$.
By definition of the controlled F\o lner function of $G/N$, one has that
there exists a finite subset $B$ of $G$ such that $B\subseteq B_{G/N,\overline{%
T}}(k)$, $\left\vert B\right\vert \leqslant k$, and $\frac{1}{\left\vert
B\right\vert }\sum_{g\in B_{G,S}\left( 10n\right) }\left\vert \overline{g}%
B\bigtriangleup B\right\vert \leqslant \left( 10n\right) ^{-1}$. By the choice of 
$\sigma $, $A=\sigma \left( B\right) $ is a subset of $B_{G,T}\left(
k\right) $ such that $B=\overline{A}$. Consider now a function $\phi
\colon G\rightarrow \mathrm{\mathrm{Sym}}\left( A\right) $, $g\mapsto \phi _{g}$,
such that $\phi _{g}\left( a\right) =\sigma \left( \overline{ag}\right) $ if 
$\overline{ag}\in \overline{A}$ (and it is defined arbitrarily otherwise).
We have that 
\begin{equation*}
A\cdot B_{G,S}\left( 10n\right) \cdot A^{-1}\subseteq B_{G,S}\left(
2k+10n\right) \subseteq B_{G,S}\left( 20k\right)
\end{equation*}%
and hence 
\begin{equation*}
N\cap A\cdot B_{G,S}\left( 10n\right) \cdot A^{-1}\subseteq B_{N,R}\left(
\Delta _{N\leqslant G}\left( 20k\right) \right) \text{.}
\end{equation*}%
By definition of $\mathcal{D}_{N,R}^{\mathcal{F}}$, we deduce that there
exist $\alpha \in I$ such that $k_{\alpha }\leqslant \mathcal{D}_{N,R}^{\mathcal{F%
}}\left( \Delta _{N\leqslant G}\left( 20k\right) \right) $, and a $\left(
20k,1\right) $-approximation $\psi\colon N\rightarrow G_{\alpha }$. By hypothesis on $\mathcal F$,
there exists $\beta \in I$ such that 
\begin{equation*}
k_{\beta }\leqslant \left\vert A\right\vert k_{\alpha }\leqslant k\mathcal{D}_{N,R}^{%
\mathcal{F}}\left( \Delta _{N\leqslant G}\left( 20k\right) \right)
\end{equation*}%
and $G_{\beta }$ is isometrically isomorphic to $\mathrm{Sym}\left( A\right)
\ltimes G_{\alpha }^{A}$. Define now the function $\Upsilon\colon G\rightarrow \mathrm{Sym}%
\left( A\right) \ltimes G_{\alpha }^{A}$ by setting $\Upsilon\left( g\right)
=\left( \psi _{g},b_{g}\right) $, where $b_{g}\in G_{\alpha }^{A}$ is
defined by $a\mapsto \psi \left( \sigma \left( ag^{-1}\right) ga\right) $.
Then the argument at the end of the proof \cite[Theorem 5.1]%
{holt_closure} shows that $\Upsilon$ is an $\left( n,1\right) $%
-approximation for $G$. This concludes our proof.
\end{proof}

\begin{corollary}
Suppose that $G$ is a finitely-generated group, and $N$ is a
finitely-generated normal subgroup of $G$ such that the quotient $G/N$ is
amenable. Then we have $$\mathcal{D}_{G}^{\emph{sof}}\left( n\right)
\preccurlyeq \mathrm{F\o l}_{G/N}^{\emph{con}}\left( n\right) \mathcal{D}%
_{N}^{\emph{sof}}(\Delta _{N\leqslant G}(\mathrm{F\o l}_{G/N}^{\emph{con}%
}\left( n\right) )).$$
\end{corollary}

\subsection{Restricted wreath products}

Let $G$ and $H$ be two groups. The \emph{regular restricted wreath product }$G\wr H$ is the
semidirect product $B\rtimes H$, where $B=\bigoplus_HG$ is the group  of
finitely-supported functions from $H$ to $G$, and the action $\upsilon\colon%
H\curvearrowright B$ is the Bernoulli shift. An element of $G\wr H$ can be
represented by a pair $\left( b,h\right) $ where $h\in H$ and $b\in B, b\colon H\to G$. It is clear
that if $G$ and $H$ are finite, then $G\wr H$ is finite and $\left\vert G\wr
H\right\vert = \left\vert H\right\vert \left\vert G\right\vert
^{\left\vert H\right\vert }$.

Suppose now that $G$ is a bi-invariant metric group and $F$ is a finite
group.\ Then the wreath product $G\wr F=B \rtimes F$ is endowed with a
canonical bi-invariant metric, defined in \cite[Section 3]{holt_closure}
as follows. For $x_{0},x_{1}\in F$ and $b_{0},b_{1}\in B,$ we set%
\begin{equation*}
d_{G\wr F}\left( \left( b_{0}, x_{0}\right) ,\left( b_{1}, x_{1}\right)
\right) =\left\{ 
\begin{array}{ll}
\max_{x\in F}d_{G}\left( b_{0}\left( x\right) ,b_{1}\left( x\right) \right)
& \text{if }x_{0}=x_{1}\text{,} \\ 
1 & \text{otherwise.}%
\end{array}%
\right.
\end{equation*}%
It is proved in \cite[Lemma 3.2]{holt_closure} that if $G$ is endowed
with a commutator-contractive invariant length function, then for any finite
group $F$ the bi-invariant metric $d_{G\wr F}$ on $G\wr F$ described above is
commutator-contractive.

\subsubsection{The metric profile of wreath products by residually finite groups}

Let $\mathcal F$ be an approximating family  $\mathcal{F}=\{(G_{\alpha},d_{\alpha}, k_{\alpha}, \varepsilon _{\alpha })\}_{\alpha\in I}.$ 
We assume that $\varepsilon _{\alpha }=1$ for every $\alpha \in I$, and furthermore
that for any $\alpha \in I$ and any finite group $F$ there exists $\beta \in
I$ such that $k_{\beta }\leqslant \left\vert F\right\vert k_{\alpha }^{\left\vert
F\right\vert }$ and $G_{\beta }$ is isometrically isomorphic to the wreath
product $G_{\alpha }\wr F$ endowed with the bi-invariant metric $d_{G\wr F}$ described above. 
These assumptions are fulfilled for example for  $\mathcal{F}\in\{ \mathcal F^{fin},  \mathcal F^{fin}_{cc},  \mathcal F^{fin}_{\{0,1\}} \}.$ 

Recall that $\Phi_{G,S}\left( n\right)$ denotes the full residual finiteness growth function of a residually finite group $G$,
see Section~\ref{sec:lea}.

\begin{theorem}
\label{Proposition:wreath-residually} Let $\mathcal{F}$ be a family
as above. Let $G$ be an group with finite generating set $R$, and $H$ be a
group with finite generating set $T$. Consider the finite generating set $S$
of $G\wr H=\bigoplus_{H}G\rtimes H$ consisting of the pairs $\left(
b,h\right) $ where $h\in H$ and $b\in \bigoplus_{H}G$ has support contained
in $\left\{ e_H\right\} $ and range contained in $R$. Then for every $n\in 
\mathbb{N}$ we have:
$$\mathcal{D}_{G\wr
H,S}^{\mathcal{F}}\left( n\right) \leqslant \Phi_{H,T}\left( 4n\right)
\cdot \mathcal{D}_{G,R}^{\mathcal{F}}\left( \Phi_{H,T}\left(
4n\right) \right) ^{\Phi_{H,T}\left( 4n\right) }$$
\end{theorem}

\begin{proof}
Clearly, we can assume that $G\ $is $\mathcal{F}$-approximable and $H$ is
residually finite, otherwise there is nothing to prove. Observe that if $%
n\in \mathbb{N}$ then any element of $B_{G\wr H,S}\left( n\right) $ is of
the form $\left( b,h\right) $, where $h\in B_{H,T}\left( n\right) $ and $%
b\in \bigoplus_{H}G$ has support contained in $B_{H,T}\left( n\right) $ and
range contained in $B_{G,R}\left( n\right) $. Set $m=\Phi%
_{H,T}\left( 4n\right) $. Then by definition of $\Phi_{H,T}$ there
exists a normal subgroup $N$ of $H$ of index at most $m$ such that $N\cap
B_{H,T}\left( 4n\right) =\left\{ e_H\right\} $. Also by definition of $%
\mathcal{D}_{G,R}^{\mathcal{F}}$ there exists $\alpha \in I$ and a $\left(
m,1\right) $-approximation $\phi\colon G\rightarrow G_{\alpha }$ such that $%
k_{\alpha }\leqslant \mathcal{D}_{G}^{\mathcal{F}}\left( m\right) $. By
hypothesis on $\mathcal F$, there exists $\beta \in I$ such that $k_{\beta }\leqslant m\cdot
k_{\alpha }^{m}\leqslant m\cdot \mathcal{D}_{G}^{\mathcal{F}}\left( m\right) ^{m}$
and $G_{\beta }$ is isometrically isomorphic to $G_{\alpha }\wr H/N$. Define
now the function $\psi\colon G\wr H\rightarrow G_{\alpha }\wr H/N$ by setting $%
\psi \left(g,h\right) =(\widehat{g}, hN)$ where $\widehat{g}\colon H/N\rightarrow
G_{\alpha }$ is defined by 
\begin{equation*}
\widehat{g}\left( kN\right) =\left\{ 
\begin{array}{ll}
e_\alpha & \text{if }B_{H,T}\left( n\right) \cap kN=\varnothing \text{, and} \\ 
\varphi \left( g_{k^{\prime }}\right) & \text{if }B_{H,T}\left( n\right)
\cap kN=\left\{ k^{\prime }\right\} \text{.}%
\end{array}%
\right.
\end{equation*}%
Then the proof of \cite[Theorem 3.1]{holt_closure} shows that $\psi $
is a $\left( n,1\right) $-approximation of $G$.
\end{proof}

\subsubsection{The sofic profile of wreath products of sofic groups by sofic
groups}

In this section, we suppose that $G$ and $H$ are groups with finite
generating sets $R$ and $T$, respectively. We let $S$ be the set of elements
of $G\wr H=\bigoplus_{H}G\rtimes H$ of the form $\left( b,h\right) $ where $%
h\in T$ and $b\in \bigoplus_{H}G$ is such that the support of $b$ is
contained in $\left\{ e_H\right\} $ and the range is contained in $R$. 

Let $K$ denote a group with a bi-invariant distance $d$. If $F$ is a finite subset of $G$ and $c,\varepsilon >0$%
, then a function $\phi\colon G\rightarrow K$ is \emph{$\left(
F,\varepsilon \right) $-multiplicative} if $d\left( \phi \left( xy\right)
,\phi \left( x\right) \phi \left( y\right) \right) <\varepsilon $ for $%
x,y\in F,$ and \emph{$\left( F,c\right) $-injective} if $d\left( \phi \left(
x\right) ,\phi \left( y\right) \right) \geqslant c$ for $x,y\in F$ distinct, cf. terminology of Definition~\ref{def:app}.
Proceeding  as in the proof of \cite[Lemma 2.8]{hayes_metric_2018} gives the following.

\begin{lemma}
\label{Lemma:wreath} Suppose that $\Psi\colon G\wr H\rightarrow K$ is a function.
Fix $n\in \mathbb{N}$ and let $G_{n}$ be the set of elements of $
\bigoplus_{H}G$ whose support is contained in $B_{H,T}\left( n\right) $ and
whose range is contained in $B_{G,R}\left( n\right) $. Suppose that $\Psi $
satisfies the following (the identities $e_{\bigoplus_{H}G}$ and $e_{H}$ are denoted by $1$):

\begin{itemize}
\item $d\left( \Psi \left( xy,1\right) ,\Psi \left( x,1\right)\Psi \left(
y,1\right) \right) <\varepsilon _{1}$ whenever $x,y\in G_{n}$,

\item $d\left( \Psi \left( 1,x\right) \Psi \left( 1,y\right) ,\Psi \left(
1, xy\right) \right) <\varepsilon _{0}$ whenever $x,y\in B_{H,T}\left(
n\right) $,

\item $\Psi \left( x,1\right) \Psi \left( 1,y\right) =\Psi \left(\upsilon
_{y}\left( x\right), y \right) $ whenever $x\in G_{n}$ and $y\in B_{H,T}\left(
n\right) $, and

\item $\Psi \left( 1, y\right) \Psi \left( x,1\right) =\Psi \left( x, y\right) 
$.
\end{itemize}

Then $d\left( \Psi \left( zw\right) ,\Psi \left( z\right) \Psi \left(
w\right) \right) <\varepsilon _{0}+\varepsilon _{1}$ for any $z,w\in B_{G\wr
H,S}\left( n\right) $.
\end{lemma}

Suppose that $K$ is a bi-invariant metric group. Let $d^{\prime }$ be any
bi-invariant metric on $\bigoplus_{B}K$ that restricts to the original
metric on $K$ on each copy of $K$, and $d^{\prime }$ is also the
corresponding metric on $\bigoplus_{B}K\wr _{B}\mathrm{Sym}\left( B\right) $%
. Consider also the maximum metric $d_{\max }$ on $\bigoplus_{B}K$ and the
corresponding metric on $\bigoplus_{B}K\wr _{B}\mathrm{Sym}\left( B\right) $%
. The proof of \cite[Proposition 3.3]{hayes_metric_2018} gives the following.

\begin{lemma}
\label{Lemma:computation}Fix $\varepsilon >0$. Suppose that $\sigma
\colon H\rightarrow \mathrm{\mathrm{Sym}}\left( B\right) $ is a $\left(
B_{H,T}\left( 4n\right) ,\varepsilon \right) $-multiplicative and $\left(
B_{H,T}\left( 4n\right) ,1-\varepsilon \right) $-injective function. Suppose
also that $\theta\colon G\rightarrow K$ is a $\left( B_{G,R}\left( 4n\right)
,\varepsilon \right) $-multiplicative and $\left( B_{G,R}\left( 4n\right)
,1-\varepsilon \right) $-injective function. Then there exists a function $$%
\Psi\colon G\wr H\rightarrow \bigoplus_{B}K\wr _{B}\mathrm{Sym}\left( B\right) $$
that is $(B_{G\wr S}\left( n\right) ,48\left\vert B_{H,T}\left( 4n\right)
\right\vert ^{2}\varepsilon )$-multiplicative with respect to the metric $%
d^{\prime }$ and $(B_{G\wr S}\left( n\right) ,1-48\left\vert B_{H,T}\left(
n\right) \right\vert ^{2}\varepsilon )$-injective with respect to the metric 
$d_{\max }$.
\end{lemma}

A quantitative analysis of the proof of \cite[Theorem 4.1]{hayes_metric_2018} then shows the
following.

\begin{lemma}
\label{Lemma:sofic-computation}Fix $\varepsilon >0$. Suppose that $\sigma
\colon H\rightarrow \mathrm{\mathrm{Sym}}\left( B\right) $ is a $\left(
B_{H,T}\left( 4n\right) ,\varepsilon \right) $-multiplicative and $\left(
B_{H,T}\left( 4n\right) ,1-\varepsilon \right) $-injective function. Suppose
that $\theta\colon G\rightarrow \mathrm{Sym}\left( A\right) $ is a $\left(
B_{G,R}\left( 4n\right) ,\varepsilon \right) $-multiplicative and $\left(
B_{G,R}\left( 4n\right) ,1-\varepsilon \right) $-injective function. Let $$%
\Psi\colon G\wr H\rightarrow \bigoplus_{B}\mathrm{Sym}\left( A\right) \wr _{B}%
\mathrm{Sym}\left( B\right) $$ be obtained from $\sigma $ and $\theta $ as in
Lemma \ref{Lemma:computation}. Define \begin{equation*} \Theta\colon \bigoplus_{B}\mathrm{Sym}%
\left( A\right) \wr _{B}\mathrm{Sym}\left( B\right) \rightarrow \mathrm{Sym}%
\left( A^{B}\times B\right) \hbox{ by }\end{equation*} \begin{equation*}
\Theta \left( \pi ,\tau \right)\colon \left( \left( a_{\beta }\right) ,b\right)
\mapsto \left( \left( \pi _{b,\beta }\left( a_{\beta }\right) \right) ,\tau
\left( b\right) \right) \text{.}
\end{equation*}%
Then the composition $\Theta \circ \Psi\colon G\wr H\rightarrow \mathrm{Sym}%
\left( B\times A^{B}\right) $ is $(B_{G\wr S}\left( n\right) ,48\left\vert
B_{H,T}\left( 4n\right) \right\vert ^{2}\varepsilon )$-multiplicative and $%
(B_{G\wr S}\left( n\right) ,1-48\left\vert B_{H,T}\left( n\right)
\right\vert ^{2}\varepsilon )$-injective when $\mathrm{Sym}\left( B\times
A^{B}\right) $ is endowed with the normalized Hamming distance.
\end{lemma}

We extract from Lemma \ref{Lemma:sofic-computation} the following upper
bound on the sofic profile of the wreath product of sofic groups. Recall
that the sofic profile $\mathcal{D}_{G,R}^{sof}\left( n\right) $ of a group $%
G\ $with finite generating set $R$ is the $\mathcal{F}$-profile $%
\mathcal{D}_{G,R}^{\mathcal{F}}\left( n\right) $ where $\mathcal{F}$ is the
family $\mathcal F^{sof}=\left( \Sym\left( A\right), d_{\rm{Ham}}, \left\vert
A\right\vert ,1\right)_{n\in \mathbb N} $ where $A$ is a finite set and $\mathrm{Sym}\left(
A\right) $ is endowed with the normalized Hamming distance.

Recall that $\beta _{G,S}\left( n\right) =\left\vert
B_{G,S}\left( n\right) \right\vert $ is the growth function of $G$ with respect to a finite generating set $S$. 
We let $\beta _{G}$ denote the $\simeq $%
-equivalence class of $\beta _{G,S}$, which is independent of the choice of the
generating set $S$.

\begin{theorem}
\label{Proposition:sofic-profile-wreath}Let $G$ be an group with finite
generating set $R$, and $H$ be a group with finite generating set $T$.
Consider the finite generating set $S$ of $G\wr H=\bigoplus_{H}G\rtimes H$
consisting of the pairs $\left( b, h\right) $ where $h\in H$ and $b\in
\bigoplus_{H}G$ has support contained in $\left\{ e_H\right\} $ and range
contained in $R$. Then for every $n\in \mathbb{N}$ one has that 
\begin{equation*}
\mathcal{D}_{G\wr H,S}^{sof}\left( n\right) \leqslant \mathcal{D}%
_{H,T}^{sof}(48\beta _{H,T}\left( n\right) ^{2}n)\cdot \mathcal{D}%
_{G,R}^{sof}(48\beta _{H,T}\left( n\right) ^{2}n)^{\mathcal{D}%
_{H,T}^{sof}(48\beta _{H,T}\left( n\right) ^{2}n)}\text{.}
\end{equation*}
\end{theorem}

\subsubsection{The hyperlinear profile of the wreath product of a
hyperlinear group by a sofic group}

We adopt the preceding notations: $G$ and $H$ are groups with finite
generating sets $R$ and $T$, respectively. We let $S$ be the finite
generating set for $G\wr H$ of elements $\left( b,h\right) $ for $h\in
B_{H,T}\left( n\right) $ and $b\in \bigoplus_{H}G$ with support contained in 
$\left\{ e_H\right\} $ and range contained in $B_{G,R}\left( n\right) $.

Suppose that $B$ is a finite set. Then we denote by $\mathcal{H}_{B}$ the
finite-dimensional Hilbert space with basis $\left\{ \left\vert
b\right\rangle :b\in B\right\} $. If $\mathcal{H}$ is a finite-dimensional
Hilbert space, then we define $\mathcal{H}^{\otimes B}$ to be the tensor
product of a family of $\left\vert B\right\vert $ copies of $\mathcal{H}$
indexed by $B$. We denote by $U\left( \mathcal{H}\right) $ the group of
unitary operators on $\mathcal{H}$ equipped with the \emph{projective}
normalized Hilbert-Schmidt pseudometric $d_{\overline{\mathrm{HS}}}$ as defined in Section~%
\ref{sec:hyp}. We analyze the proof of \cite[Section 4.2]{hayes_metric_2018} and obtain
the following.

\begin{lemma}
\label{Lemma:hyperlinear-computation}Fix $\varepsilon >0$. Suppose that $%
\sigma\colon H\rightarrow \mathrm{\mathrm{Sym}}\left( B\right) $ is a $%
\left( B_{H,T}\left( 4n\right) ,\varepsilon \right) $-multiplicative and $%
\left( B_{H,T}\left( 4n\right) ,1-\varepsilon \right) $-injective function.
Suppose that $\theta\colon G\rightarrow U\left( \mathcal{H}\right) $ is a $%
\left( B_{G,R}\left( 4n\right) ,\varepsilon \right) $-multiplicative and $%
\left( B_{G,R}\left( 4n\right) ,1-\varepsilon \right) $-injective function
for some finite-dimensional Hilbert space $\mathcal{H}$. Let $\Psi\colon
G\wr H\rightarrow \bigoplus_{B}U\left( \mathcal{H}\right) \wr \mathrm{Sym}%
\left( B\right) $ be the function obtained in Lemma \ref{Lemma:computation}.
Define the function $\Theta\colon \bigoplus_{B}U\left( \mathcal{H}\right)
\wr _{B}\mathrm{Sym}\left( B\right) \rightarrow U\left( \mathcal{H}^{\otimes
B}\otimes \mathcal{H}_{B}\right) $ by%
\begin{equation*}
\Theta \left( \pi ,\tau \right)\colon \bigotimes_{\gamma }\xi _{\gamma
}\otimes \left\vert b\right\rangle \mapsto \bigotimes_{\gamma }\pi
_{b,\gamma }\left( \xi _{\gamma }\right) \otimes \left\vert \sigma \left(
b\right) \right\rangle \text{.}
\end{equation*}%
Then $\Theta \circ \Psi $ is $(B_{G\wr H, S}\left( n\right) ,48\left\vert
B_{H,T}\left( 4n\right) \right\vert ^{2}\varepsilon ^{\frac{1}{2}})$%
-multiplicative function and \linebreak $(B_{G\wr H, S}\left( n\right) ,1-48\left\vert
B_{H,T}\left( n\right) \right\vert ^{2}\varepsilon ^{\frac{1}{2}})$%
-injective function when $U\left( \mathcal{H}^{\otimes B}\otimes \mathcal{H}%
_{B}\right) $ is endowed with the normalized Hilbert-Schmidt distance.
\end{lemma}

We deduce the following.

\begin{theorem}
Let $G$ be an group with finite generating set $R$, and $H$ be a group with
finite generating set $T$. Consider the finite generating set $S$ of $G\wr
H=\bigoplus_{H}G\rtimes H$ consisting of the pairs $\left( b,h\right) $
where $h\in H$ and $b\in \bigoplus_{H}G$ has support contained in $\left\{
e_{H}\right\} $ and range contained in $R$. Then for every $n\in \mathbb{N}$
one has that 
\begin{equation*}
\mathcal{D}_{G\wr H,S}^{\overline{hyp}}\left( n\right) \leqslant \mathcal{D}%
_{H,T}^{sof}(2500\beta _{H,T}\left( n\right) ^{4}n^{2})\cdot \mathcal{D}%
_{G,R}^{\overline{hyp}}(2500\beta _{H,T}\left( n\right) ^{4}n^{2})^{\mathcal{D}%
_{H,T}^{sof}(2500\beta _{H,T}\left( n\right) ^{4}n^{2})}\text{.}
\end{equation*}
\end{theorem}

\subsubsection{The linear sofic profile of the wreath product of a linear
sofic group by a sofic group}

Let $\mathbb{K}$ be a field. If $B$ is a finite set, then we let $\mathbb{K}%
^{B}$ to be the $\mathbb{K}$-vector space obtained as the direct sum of $%
\left\vert B\right\vert $ copies of $\mathbb{K}$ indexed by $B$ with basis $%
\left\{ \left\vert b\right\rangle :b\in B\right\} $. If $V$ is a
finite-dimensional $\mathbb{K}$-vector space, then we let $V^{\otimes B}$ be
the tensor product of a family of $\left\vert B\right\vert $ copies of $V$
indexed by $B$. We denote by $GL\left( V,\mathbb{K}\right) $ the group of
invertible operators on $V$ equipped with the \emph{projective }normalized
rank pseudometric $d_{\overline{\rm rank}}$ as defined in Section~\ref{sec:linsof}. We analyze the proof of  \cite[Proposition 4.12 ]{hayes_metric_2018} and obtain
the following.

\begin{lemma}
Fix $\varepsilon >0$. Suppose that $\sigma\colon H\rightarrow \mathrm{Sym}%
\left( B\right) $ is a $\left( B_{H,T}\left( 4n\right) ,\varepsilon \right) $%
-multiplicative and $\left( B_{H,T}\left( 4n\right) ,1-\varepsilon \right) $%
-injective function. Suppose also that $\theta\colon G\rightarrow GL\left(
V, \mathbb{K}\right) $ is a $\left( B_{G,R}\left( 4n\right) ,\varepsilon
\right) $-multiplicative and $\left( B_{G,R}\left( 4n\right) ,1-\varepsilon
\right) $-injective function. Let $\Psi\colon G\wr H\rightarrow \bigoplus_{B}%
\mathrm{GL}\left( V, \mathbb{K}\right) \wr _{B}\mathrm{Sym}\left( B\right) $
be the obtained as in Lemma \ref{Lemma:computation}. Let also 
\begin{equation*}
\Theta\colon \bigoplus_{B}\mathrm{GL}\left( V, \mathbb{K}\right) \wr _{B}%
\mathrm{Sym}\left( B\right) \rightarrow \mathrm{GL}\left( V^{\otimes
B}\otimes \mathbb{K}^{B}, \mathbb{K}\right) 
\end{equation*}
be defined by%
\begin{equation*}
\Theta \left( \pi ,\tau \right)\colon \bigotimes_{\gamma }a_{\gamma }\otimes
\left\vert b\right\rangle \mapsto \bigotimes_{\gamma }\pi _{b,\gamma }\left(
a_{\gamma }\right) \otimes \left\vert \tau \left( b\right) \right\rangle 
\text{.}
\end{equation*}%
Then $\Theta \circ \Psi $ is $(B_{G\wr H, S}\left( n\right) , 48\left\vert
B_{H,T}\left( 4n\right) \right\vert ^{2}\varepsilon )$-multiplicative
function and\linebreak  $(B_{G\wr H, S}\left( n\right) ,1-48\left\vert B_{H,T}\left(
4n\right) \right\vert ^{2}\varepsilon )$-injective function.
\end{lemma}

The following is an immediate consequence of the preceding lemma.

\begin{theorem}
Let $G$ be an group with finite generating set $R$, and $H$ be a group with
finite generating set $T$. Consider the finite generating set $S$ of $G\wr
H=\bigoplus_{H}G\rtimes H$ consisting of the pairs $\left( b,h\right) $
where $h\in H$ and $b\in \bigoplus_{H}G$ has support contained in $\left\{
e_{H}\right\} $ and range contained in $R$. Then for every $n\in \mathbb{N}$
one has that 
\begin{equation*}
\mathcal{D}_{G\wr H,S}^{\overline{lin}}\left( n\right) \leqslant \mathcal{D}%
_{H,T}^{sof}(48n\beta _{H,T}\left( n\right) ^{2})\cdot \mathcal{D}%
_{G,R}^{\overline{lin}}(48n\mathcal{D}_{H,T}^{sof}(48n\beta _{H,T}\left( n\right)
^{2}))^{\mathcal{D}_{H,T}^{sof}(48n\beta _{H,T}\left( n\right) ^{2})}\text{.}
\end{equation*}
\end{theorem}

\section{Further remarks and open questions}

The following question incites, in particular, a thorough study of possible definitions of bi-invariant metrics on solvable groups;
see also our observation in Example~\ref{ex:variety}.

\begin{question}\label{q:solv}
Does there exist an infinite group which is not approximable by solvable groups (with no uniform bound on the derived length)?
\end{question}

We denote by $Th_c^\forall(G)$ the universal theory of $G$ in the continuous logic setting~\cite{ben_yaacov_model_2008}.
An affirmative answer to the next question would generalize Malcev's result, see Example~\ref{ex:constant}.

\begin{question}\label{q:cmalcev}
Let $G$ be a group such that $Th_c^\forall\left(GL(k, \mathbb K_\alpha)\right)\subseteq Th_c^\forall\left(G\right)$ for some $k>0$ and a field $\mathbb K_\alpha$. Is $G$ linear? 
\end{question}

 Answers to the next two questions will clarify the status of Conecture~\ref{conj:sofhyp}.
 
\begin{question}\label{q:sofhyp}
Does there exist a sofic group $G$ which is not $\mathcal F^{hyp}$-stable but satisfies $\mathcal{D}_{G}^{sof} (n)\preccurlyeq \mathcal{D}_{G}^{hyp}(n)$?
\end{question}

\begin{question}\label{q:sofhyps}
Does there exist a sofic group $G$ which is not $\mathcal F^{sof}$-stable but $\mathcal F^{hyp}$-stable?
\end{question}

Let $\bar U=U(R)$ be the unitary group of the hyperfinite factor $R$ of type $II_1$ equipped with the ultraweak topology. 
A group $G$ is hyperlinear if and only if $G$ embeds into a metric ultrapower of $\bar U$~\cite[Corollary 4.3]{pestov_introduction_2012}. 
 In \cite{arzhantseva_approx_2012}, the first author observed that all Gromov hyperbolic groups $G$ are residually finite (respectively, LE-$\mathcal F^{fin}$, LEA, etc.)
 if and only if $G$ embeds into $\bar U$. 

\begin{question}\label{q:upower}
Let $G$ be a non-elementary Gromov hyperbolic group such that $G\hookrightarrow \prod_{
\mathcal{U}}\left( \bar U, d\right) $. Does it imply the existence of an embedding $G\hookrightarrow \bar U$?
\end{question}
A positive answer will establish  the following conjecture.
\begin{conjecture}\cite[Conjecture 2.8]{arzhantseva_approx_2012}
All Gromov hyperbolic groups are residually finite $\Longleftrightarrow$ all Gromov hyperbolic groups are sofic.
\end{conjecture}

An answer to the next question will give a better understanding of the rank metric 
on linear groups, and hence, of linear sofic groups and their profile functions.

\begin{question}\cite{arzhantseva_linear_2012}\label{q:lsofic}
Does the class of linear sofic groups, i.e. $\mathcal F^{lin}=(GL(n, \mathbb{K}), d_{\rm{rank}}, n, 1/4)_{n\in \mathbb N}$-approximable groups,
coincide with the class of  
$(GL(n, \mathbb{K}), d_{\rm{rank}}, n, \varepsilon_n)_{n\in \mathbb N}$-approximable groups, where
$\varepsilon_n$ is  constantly equal to a fixed or to an arbitrarily chosen number between $1/4$ and $1$?
\end{question}

The next three questions are about examples of sofic and linear sofic groups with extreme profile functions with respect to
the ambient class of weakly sofic groups.

\begin{question}\label{q:qsofic}
Does there exist a sofic group $G$ such that $\mathcal{D}_{G, S}^{fin} (n)\simeq n^n$?
\end{question}

\begin{question}\label{q:qlsofic1}
Does there exist an infinite linear sofic group $G$ such that 
$\mathcal{D}_{G, S}^{fin} (n)\simeq \mathcal{D}_{G, S}^{lin}(n)$?
\end{question}

\begin{question}\label{q:qlsofic2}
What is a relationship between 
$\mathcal{D}_{G, S}^{fin} (n)$ and $\mathcal{D}_{G, S}^{lin}(n)$ for a linear sofic group $G$?
\end{question}

If affirmative, the answer to the next question will generalize the famous
Coulhon-Saloff-Coste isoperimetric inequality which is, using our notation: $\beta_{G,S}(n)\preccurlyeq\mathrm{F\o l}_{G,S}\left( n\right)$. 

\begin{question}\label{q:nilp}
Is $\beta_{G,S}(n)\preccurlyeq\mathcal{D}_{G, S}^{sof} (n)$ for a sofic group $G$?
\end{question}

To answer this rather challenging question one can first focus on subclasses of sofic groups such as classical matrix groups and
(elementary) amenable groups.  Here is a variant for groups locally embeddable into
amenable groups.

\begin{question}\label{q:ra}
Is $\beta_{G,S}(n)\preccurlyeq LEA_{G, S} (n)$ for an LEA  group $G$?
\end{question}

Recall that a finitely generated linear group that is not virtually nilpotent has exponential weakly sofic profile $\mathcal{D}_{G, S}^{fin} (n)$, see Example~\ref{ex:linear}.

\begin{question}\label{q:lin}
Does there exist a finitely generated linear group of exponential growth with polynomial/subexponential sofic profile  $\mathcal{D}_{G, S}^{sof} (n)$?
\end{question}

Corollary~\ref{cor:stab} and known stability results on virtually abelian groups, see Example~\ref{ex:stablegps}, yield the following.

\begin{conjecture}\label{conj:ablin}
Let $G$ be a finitely generated virtually abelian group. Then $\mathcal{D}_{G, S}^{lin} (n)\simeq n^{{\rm rank}\, G}$.
\end{conjecture}

A direct approach to establish Conjecture~\ref{conj:ablin}  is to proceed as in Example~\ref{ex:ab}, using the discreteness of 
values of the rank distance $d_{\rm{rank}}$. 
Alternatively, Conjecture~\ref{conj:ablin}  would be established by proving the $\mathcal F^{lin}$-stability of virtually abelian groups. This remains unknown, see~\cite{elek2017commuting} for a partial result. In this vein, it is interesting to investigate the (non)-stability of commutator relator word with respect to the rank distance within various classes of matrices. Unitary, self-adjoint and normal matrices are natural classes to consider, see Section~\ref{sec:linsof} for notation.

\begin{question}\label{q:srk}
Let $G$ be a finitely generated virtually abelian group. Is it true that $\mathcal{D}_{G, S}^{u} (n)\simeq\mathcal{D}_{G, S}^{sa} (n)\simeq\mathcal{D}_{G, S}^{nor} (n)\simeq n^{{\rm rank}\, G}$?
\end{question}

We have shown that sofic and hyperlinear profiles of the Heisenberg groups $H_{2l+1}$ coincide with the full residual finiteness growth, see Example~\ref{ex:stablegps}. We expect this to hold for the weakly sofic and linear sofic profiles as well, cf. Example~\ref{ex:linear}.

\begin{conjecture}\label{conj:heis}
We have $\mathcal{D}_{H_{2l+1}, S}^{fin} (n)\simeq \mathcal{D}_{H_{2l+1}, S}^{lin} (n)\simeq n^{2(2l+1)}$.
\end{conjecture}

Here is an ambitious generalization of Example~\ref{ex:ab} and Example~\ref{ex:stablegps}.
\begin{conjecture}\label{conj:nilp} Let $G$ be a finitely generated virtually nilpotent group. Then
$$\mathcal{D}_{G, S}^{sof} (n)\simeq\mathcal{D}_{G, S}^{hyp} (n)\simeq\mathcal{D}_{G, S}^{lin} (n)\simeq\mathcal{D}_{G, S}^{fin} (n)\simeq \Phi_{G,S}(n).$$
\end{conjecture}

To approach the weakly sofic part of the preceding conjecture from an ambient class of groups, we consider arbitrary residually finite groups.

\begin{question}\label{q:lef}
Does there exist a finitely generated residually finite group with $\mathcal{D}_{G, S}^{fin} (n)\not\simeq \Phi_{G,S}(n)$?
\end{question}

The next question explores a possibility to describe weakly sofic groups in purely algebraic terms (see~\cite{glebskyA} for such a result on sofic groups).

\begin{question}\label{q:wsA}
 Is there a characterization of weakly sofic groups with no reference to any non-trivial bi-invariant distances $d_\alpha$ on finite approximating groups $H_\alpha$?
 \end{question}

In Section~\ref{sec:direct}, we give an upper bound on the sofic profile of a free product of sofic groups. Other profiles of free products remain unexplored. Recall that $F_2$ denotes a free group of rank~2.
 
\begin{conjecture}\label{conj:lsfree}
Let $G$ and $H$ be linear sofic groups. Then 
$$\mathcal{D}^{lin}_{G\ast H}\left(
n\right) \preccurlyeq \left(\mathcal{D}^{lin}_{G}\left( n\right) + \mathcal{D}^{lin}_{H}\left(
n\right)\right) \Phi_{F_{2}}\left( n\right).$$

\end{conjecture}

\begin{question}\label{q:hypfree} 
Let $G$ and $H$ be hyperlinear groups. Find an upper bound on $\mathcal{D}^{hyp}_{G\ast H}\left(
n\right)$ in terms of  $\mathcal{D}^{hyp}_{G}\left( n\right)$ and  $\mathcal{D}^{hyp}_{H}\left(
n\right)$. Is it true that
$$\mathcal{D}^{hyp}_{G\ast H}\left(
n\right) \preccurlyeq \left(\mathcal{D}^{hyp}_{G}\left( n\right) + \mathcal{D}^{hyp}_{H}\left(
n\right)\right) \Phi_{F_{2}}\left( n\right)?$$
 \end{question}

\begin{question}\label{q:leffp}
Let $G$ and $H$ be LE-$\mathcal F$ groups. Estimate $\Phi^{\mathcal F}_{G\ast H}\left(
n\right)$ in terms of  $\Phi_G^{\mathcal F}\left( n\right)$ and  $\Phi^{\mathcal F}_{H}\left(
n\right)$. 

\end{question}
\begin{question}\label{q:leafp}
Let $G$ and $H$ be LEA groups. Estimate $LEA_{G\ast H}\left(
n\right)$ in terms of  $LEA_G\left( n\right)$ and  $LEA_{H}\left(
n\right)$.

\end{question}

\begin{question}\label{q:sln}
Let $\mathcal{F}\in\{ \mathcal F^{sof}, \mathcal F^{hyp}, \mathcal F^{lin}, \mathcal F^{fin}, \mathcal F^{ct},  \mathcal F^{fin}_{\{0,1\}} \}$ and $m\geqslant 3$. What is $\mathcal{D}_{SL_m(\mathbb Z), S}^{\mathcal F} (n)$?
\end{question}

\bibliographystyle{amsplain}
\bibliography{quantify}

\end{document}